\documentclass[a4paper, reqno, 11pt]{amsart}
\makeatletter

\@addtoreset{equation}{section}
\makeatother
\usepackage{setspace}
\usepackage{xcolor}
\usepackage{braket}
\usepackage{mathtools}
\usepackage{amsmath}
\usepackage{amssymb}
\usepackage{latexsym}
\usepackage{amsthm}
\usepackage{mathrsfs}
\usepackage{textcomp}
\usepackage{hyperref}
\usepackage{here}
\everymath{\displaystyle}
\newtheorem{Thm}{Theorem}[section]

\newtheorem{Lem}[Thm]{Lemma}
\newtheorem{Prop}[Thm]{Proposition}

\theoremstyle{definition}

\newtheorem{Def}[Thm]{Definition}
\newtheorem{Rem}[Thm]{Remark}
\newtheorem{Exa}[Thm]{Example}

\begin{document}

\title[]{Quantitative estimates for singularity for conjugate equations driven by linear fractional transformations}
\author[]{Kazuki Okamura}
\date{\today}
\address{Department of Mathematics, Faculty of Science, Shizuoka University}
\email{okamura.kazuki@shizuoka.ac.jp}
\subjclass[2020]{39B12, 26A30, 60G30}
\maketitle

\begin{abstract}
We consider the conjugate equation driven by two families of finite maps on the unit interval satisfying a compatibility condition. 
This framework contains de Rham's functional equations. 
We give sufficient conditions for  singularity of the solution with quantitative estimates in the case where the equation is driven by a family of non-affine maps and a family of linear fractional transformations. 
\end{abstract}

\section{Introduction}

Let $N \ge 2$ and let $\mathcal{I}_N \coloneqq \{0,1, \ldots, N-1\}$. 
Let 
$\{f_i\}_{i \in \mathcal{I}_N}$ and $\{g_i\}_{i \in \mathcal{I}_N}$ be two families of $N$ maps on $[0,1]$ satisfying a compatibility condition. 
We consider the following conjugate equation for $\varphi \colon [0,1] \to [0,1]$:   
\begin{equation}\label{eq:gen-dR-def-1}
g_i \circ \varphi = \varphi \circ f_i, \ i \in \mathcal{I}_N. 
\end{equation}
See the beginning of Section 2 below for the definition of the compatibility condition. 
An alternative expression is given by 
\begin{equation}\label{eq:gen-dR-def-2}
\varphi(x) = \begin{cases} g_0 (\varphi(f_0^{-1}(x))) & 0 = f_0 (0) \le x \le f_0 (1)  \\ \vdots & \\ g_{N-1} (\varphi(f_{N-1}^{-1}(x))) & f_{N-1}(0) \le x \le f_{N-1} (1) = 1 \end{cases}.  
\end{equation}
This is known as de Rham's functional equation \cite{deRham1956, deRham1957}. 
This functional equation is generally fractal in nature, and solutions of \eqref{eq:gen-dR-def-1} can be singular functions.  

Assume that $N=2$, $0 < p < 1$, $f_0 (x) = x/2$, $f_1 (x) = (x+1)/2$, $g_0 (x) = px$ and $g_1 (x) = (1-p)x + p$. 
Then, \eqref{eq:gen-dR-def-2} is expressed as 
\[ \varphi(x) = \begin{cases} p \varphi(2x) & 0 \le x \le 1/2  \\  (1-p) \varphi(2x-1) + p & 1/2 \le x \le 1 \end{cases}. \]
There exists a unique continuous strictly increasing solution $\varphi_p$ for this equation, 
and  the function $\varphi_p$ is singular if $p \ne 1/2$, and $\varphi_p (x) = x$ if $p=1/2$. 
See \cite[Subsection 3.4]{YHK1997} for the proof. 
The function $\varphi_p$ is often called the Lebesgue singular function. 
This function is sometimes referred to by the name of Riesz and N\'agy (\cite{Allaart2018, Baek2011}), Salem (\cite{dACFS2017}), or Tak\'acs (\cite{Baek2011}). 
This function is also a solution of the functional equation by Matkowski and Weso{\l}owski \cite{Morawiec2018, Morawiec2019}. 

Minkowski's question-mark function was initially considered in \cite{Minkowski1905}. 
This function can be obtained by a solution of \eqref{eq:gen-dR-def-1} for the case where $f_0 (x) = x/(x+1)$, $f_1 (x) = 1/(2-x)$, $g_0 (x) = x/2$, and $g_1 (x) = (x+1)/2$. 
It can also be expressed in the form 
\[ \varphi(x) = \begin{cases} \frac{1}{2} \varphi(\frac{x}{1-x}) & 0 \le x \le 1/2  \\  \frac{1}{2} \varphi\left(2-\frac{1}{x}\right) + \frac{1}{2} & 1/2 \le x \le 1 \end{cases}. \]
This function has appeared in various areas of mathematics such as number theory and dynamical systems.  
The real analytic properties of this function have been considered (\cite{Jordan2016, Kessebohmer2008, Mantica2017}). 
The conjugate equation of the de Rham type has been investigated by many authors (\cite{Barany2018, Berg2000, Buescu2021, Girgensohn2006, Hata1985K, Hata1985, Kawamura2002, Serpa2015, Serpa2015N, Serpa2017, Zdun2001}).  
We refer the reader to \cite[Subsection 3.4]{Okamura2020} for a more detailed review of de Rham's functional equations.   

The main purpose of this paper is to investigate sufficient conditions for the solution of \eqref{eq:gen-dR-def-1} to be singular. 
Necessary and sufficient conditions of $\{f_i\}_{i \in \mathcal{I}_N}$ and $\{g_i\}_{i \in \mathcal{I}_N}$ for the singularity of the solution have been considered in \cite{Barany2018, BenSlimane2008, Girgensohn1993, Hata1985, Kawamura2011, Morawiec2019}.  
We consider the case where both  $\{f_i\}_{i \in \mathcal{I}_N}$ and $\{g_i\}_{i \in \mathcal{I}_N}$ are {\it non-affine} maps. 
This case is more difficult to analyze than the case where both $\{f_i\}_{i \in \mathcal{I}_N}$ and $\{g_i\}_{i \in \mathcal{I}_N}$ are affine maps. 
In \cite{Okamura2020}, the author dealt with the case where $\{g_i\}_{i \in \mathcal{I}_N}$ are linear fractional transformations and $\{f_i\}_{i \in \mathcal{I}_N}$ are affine maps. 
Our results, Theorems \ref{thm:main-singular-1} and \ref{thm:main-singular-2} below, deal with the case where $\{f_i\}_{i \in \mathcal{I}_N}$ are not necessarily affine maps. 
We do not assume the differentiability of $\{f_i\}_{i \in \mathcal{I}_N}$. 
Since $\varphi$ is an increasing function on $[0,1]$, 
we deal with the probability measure $\mu_{\varphi}$, which is the Borel probability measure on $[0,1]$ whose distribution function is $\varphi$. 
We give sufficient conditions for $\dim_H \mu_{\varphi} < 1$, which implies the singularity for $\varphi$. 
For the proof, we use an approach using martingales as in \cite{Okamura2020}, but here we need more delicate quantitative estimates for $\dim_H \mu_{\varphi}$.  

The structure of this paper is as follows. 
In Section \ref{sec:prelim}, we give some terminology and basic results of the conjugate equation on the unit interval.   
In Section \ref{sec:singu-non-linear}, we give our main results and their proofs.  
Finally, we give three numerical examples in Section \ref{sec:ex}. 

\section{Preliminaries}\label{sec:prelim}

\subsection{Existence and uniqueness}\label{sec:exist-unique}

We first discuss the existence and uniqueness of the solution of the conjugate equation \eqref{eq:gen-dR-def-1}.   

We say that a family of maps $\{h_i\}_{i \in \mathcal{I}_N}$ on $[0,1]$ is a {\it compatible system} or that it satisfies a {\it compatibility condition} if 
they are strictly increasing continuous functions such that $h_0 (0) = 0$, $h_{i-1}(1) = h_i (0)$, $1 \le i \le N-1$, $h_{N-1}(1) = 1$, and 
that a compatible system $\{h_i\}_{i \in \mathcal{I}_N}$ is a  {\it D-system} on $[0,1]$ 
if  $$\mathcal{D}_h \coloneqq \left\{h_{i_1} \circ \cdots \circ h_{i_n}(j) \colon i_1, \ldots,  i_n \in \mathcal{I}_N, j \in \{0,1\} \right\}$$ is dense in $[0,1]$.

\begin{Lem}\label{lem:density-consequence}
Let $\{h_i\}_{i \in \mathcal{I}_N}$ be a D-system on $[0,1]$. 
Then, for every $x \in [0,1]$, there exists an infinite sequence $(i_n)_n$ of $\mathcal{I}_N$ such that 
\begin{equation}\label{eq:real-address} 
x = \lim_{n \to \infty} h_{i_1} \circ \cdots \circ h_{i_n} (0) = \lim_{n \to \infty} h_{i_1} \circ \cdots \circ h_{i_n} (1).  
\end{equation}
If $x \in \mathcal{D}_h \setminus \{0,1\}$, then there exist two such infinite sequences. 
Otherwise, there exists exactly one such infinite sequence. 
\end{Lem}

\begin{proof}
We first show that for every infinite sequence $(i_n)_n$, 
\begin{equation}\label{eq:density-arbitrarily-close}
\lim_{n \to \infty} |h_{i_1} \circ \cdots \circ h_{i_n} (0) - h_{i_1} \circ \cdots \circ h_{i_n} (1)| = 0.  
\end{equation}
Assume this fails. 
By the definition of $\{h_i\}_{i \in \mathcal{I}_N}$, 
the sequence $\left( h_{i_1} \circ \cdots \circ h_{i_n} (0) \right)_n$ is increasing and $\left( h_{i_1} \circ \cdots \circ h_{i_n} (1) \right)_n$ is decreasing. 
Since $\mathcal{D}_h$ is dense in $[0,1]$, 
there exist $M \in \mathbb{N}$ and  $k_1, \ldots, k_M$ as well as $j \in \{0,1\}$ such that 
\[ h_{i_1} \circ \cdots \circ h_{i_M}(0) \le \lim_{n \to \infty} h_{i_1} \circ \cdots \circ h_{i_n}(0) < h_{k_1} \circ \cdots \circ h_{k_M}(j) \] 
\[ < \lim_{n \to \infty} h_{i_1} \circ \cdots \circ h_{i_n}(1) \le h_{i_1} \circ \cdots \circ h_{i_M}(1). \]
This cannot occur since the sequence $\left( h_{i_1} \circ \cdots \circ h_{i_n}(j) \right)_{i_1, \dots, i_n, j}$ is in lexicographic order, more precisely, $h_{i_1} \circ \cdots \circ h_{i_n}(j) \le h_{i_1^{\prime}} \circ \cdots \circ h_{i_n^{\prime}}(j^{\prime})$ if and only if $\displaystyle j + \sum_{\ell = 1}^{n} i_{\ell} N^{n-\ell} \le j^{\prime} + \sum_{\ell = 1}^{n} i_{\ell}^{\prime} N^{n-\ell}$. 
Hence \eqref{eq:density-arbitrarily-close} holds.

{\it Case 1.}  Assume $x \in \mathcal{D}_h \setminus\{0,1\}$. 
Assume that $ x = h_{i_1} \circ \cdots \circ h_{i_n}(0)$. 
Since $h_0 (0) = 0$ and $x > 0$, we can assume that $i_n \ge 1$. 
Let $i_k \coloneqq 0$ for every $k \ge n+1$. 
Then, the first equality of \eqref{eq:real-address} holds. 
By this and \eqref{eq:density-arbitrarily-close}, 
the second equality of \eqref{eq:real-address} holds. 
Thus \eqref{eq:real-address} holds.
By replacing $i_n$ with $i_n - 1$ and letting $i_k \coloneqq N-1$ for every $k \ge n+1$, 
the equations in \eqref{eq:real-address} hold. 
Assume that $ x = h_{i_1} \circ \cdots \circ h_{i_n}(1)$. 
Since $h_{N-1} (1) = 1$ and $x < 1$, we can assume that $i_n \le N-2$. 
Let $i_k \coloneqq N-1$ for every $k \ge n+1$. 
Then, \eqref{eq:real-address} holds.

It remains valid if we replace $i_n$ with $i_n + 1$ and let $i_k \coloneqq 0$ for every $k \ge n+1$. 
There are no more solutions of the infinite sequences satisfying \eqref{eq:real-address} since the sequence $\left( h_{i_1} \circ \cdots \circ h_{i_n}(j) \right)_{i_1, \dots, i_n, j}$ is in lexicographic order. 

{\it Case 2.}  Assume $x \in \{0,1\}$.  
Assume $x = 0$. 
Since each $h_i$ is strictly increasing, 
there is no other choice than $i_k = 0$ for every $k \ge 1$ for \eqref{eq:real-address}. 
Assume $x=1$. 
Then, there is no other choice than $i_k = N-1$ for every $k \ge 1$ for \eqref{eq:real-address}. 
 
{\it Case 3.}  Assume $x \notin  \mathcal{D}_h$. 
Then, for each $n \ge 1$, there exists a unique finite sequence $(i_{n,1}, \ldots, i_{n,n})$ such that 
\[ h_{i_{n,1}} \circ \cdots \circ h_{i_{n,n}} (0) < x < h_{i_{n,1}} \circ \cdots \circ h_{i_{n,n}} (1). \]
Then, $i_{n-1,k} = i_{n,k}$ for every $k \le n-1$. 
Let $i_k \coloneqq  i_{n,k}$. 
By  \eqref{eq:density-arbitrarily-close}, the equations in \eqref{eq:real-address} hold. 
If \eqref{eq:real-address} holds, then for every $M \ge 1$, 
\[ h_{i_1} \circ \cdots \circ h_{i_M}(0) \le \lim_{n \to \infty} h_{i_1} \circ \cdots \circ h_{i_n}(0) = x = \lim_{n \to \infty} h_{i_1} \circ \cdots \circ h_{i_n}(1) \le h_{i_1} \circ \cdots \circ h_{i_M}(1). \]
Hence, 
\[ h_{i_1} \circ \cdots \circ h_{i_M}(0) <  x < h_{i_1} \circ \cdots \circ h_{i_M}(1). \]
Hence the uniqueness holds. 
\end{proof}

\begin{Rem}
The existence and uniqueness of the bounded solutions of conjugate equations on more general spaces than $[0,1]$ have been discussed in \cite{Buescu2021, Okamura2019, Serpa2015, Serpa2017}. 
\end{Rem}

\begin{Lem}\label{lem:exist-unique-gen-dR}
Assume that $\{f_i\}_{i \in \mathcal{I}_N}$ and $\{g_i\}_{i \in \mathcal{I}_N}$ are D-systems on $[0,1]$. 
Then, there exists a unique continuous strictly increasing solution of \eqref{eq:gen-dR-def-1}. 
\end{Lem}

Hereafter, we let $\mu_{\varphi}$ be the continuous probability measure on $[0,1]$ whose distribution function is $\varphi$.  
This is the central object considered in this paper. 

\begin{proof}
First we show the existence. 

{\it Step 1.} We define $\varphi(x)$ for $x \in \mathcal{D}_{f}$.  
Let 
\begin{equation}\label{eq:def-solution-dyadic}
\varphi\left( f_{i_1} \circ \cdots \circ f_{i_n}(j) \right) \coloneqq g_{i_1} \circ \cdots \circ g_{i_n}(j), \ i_1, \ldots, i_n \in \mathcal{I}_N, \ j \in \{0,1\}. 
\end{equation} 
This is well-defined since $\{f_i\}_i$ and $\{g_i\}_i$ are compatible systems. 
Hence, $\varphi(\mathcal{D}_{f}) = \mathcal{D}_{g}$. 

{\it Step 2.} We define $\varphi(x)$ for $x \notin \mathcal{D}_{f}$.  
By the assumption and Lemma \ref{lem:density-consequence}, 
there exists a unique infinite sequence $(i_{n})_n$ of $\mathcal{I}_N$ such that 
$x \in (f_{i_1} \circ \cdots \circ f_{i_n}(0), f_{i_1} \circ \cdots \circ f_{i_n}(1))$. 
Since $\left[ g_{i_1} \circ \cdots \circ g_{i_n}(0), g_{i_1} \circ \cdots \circ g_{i_n}(1)\right] $ is decreasing with respect to $n$, 
\[ \bigcap_{n \ge 1} \left[ g_{i_1} \circ \cdots \circ g_{i_n}(0), g_{i_1} \circ \cdots \circ g_{i_n}(1)\right] = \left[ \lim_{n \to \infty} g_{i_1} \circ \cdots \circ g_{i_n}(0), \lim_{n \to \infty} g_{i_1} \circ \cdots \circ g_{i_n}(1) \right].  \]
By the density of $\mathcal{D}_g$ in $[0,1]$, 
we can apply \eqref{eq:density-arbitrarily-close}, and obtain that 
\begin{equation*}\label{eq:equal-dense}
\lim_{n \to \infty} g_{i_1} \circ \cdots \circ g_{i_n}(0) = \lim_{n \to \infty} g_{i_1} \circ \cdots \circ g_{i_n}(1). 
\end{equation*} 
We let $\displaystyle \varphi(x) \coloneqq \lim_{n \to \infty} g_{i_1} \circ \cdots \circ g_{i_n}(0)$. 

{\it Step 3.} We show that $\varphi$ is strictly increasing on $[0,1]$. 
By \eqref{eq:def-solution-dyadic}, 
$\varphi$ is strictly increasing on $\mathcal{D}_{f}$. 
Let $x < y$. 
By the density of $\mathcal{D}_f$ in $[0,1]$, 
we can apply \eqref{eq:density-arbitrarily-close}, and obtain that 
\begin{equation*}\label{eq:equal-dense-f}
\lim_{n \to \infty} f_{i_1} \circ \cdots \circ f_{i_n}(0) = \lim_{n \to \infty} f_{i_1} \circ \cdots \circ f_{i_n}(1). 
\end{equation*} 

Hence it holds that for some $n$ and $i_1, \dots, i_n, j_1, \dots, j_n \in \mathcal{I}_N$, 
\[  f_{i_1} \circ \cdots \circ f_{i_n}(0)\le x \le f_{i_1} \circ \cdots \circ f_{i_n}(1) < f_{j_1} \circ \cdots \circ f_{j_n}(0) \le y \le f_{j_1} \circ \cdots \circ f_{j_n}(1).  \] 
By this and Step 2, 
\[ \varphi(x) \le g_{i_1} \circ \cdots \circ g_{i_n}(1) < g_{j_1} \circ \cdots \circ g_{j_n}(0) \le \varphi(y). \]

{\it Step 4.}  We show that $\varphi$ is continuous on $[0,1]$. 
Assume that there is a discontinuous point $x_0 \in (0,1)$ of $\varphi$. 
Then, 
\[ b_0 \coloneqq \lim_{x \to x_0 + 0} \varphi(x) > \lim_{x \to x_0 - 0} \varphi(x) \eqqcolon a_0. \] 
By Step 3, $\varphi([0,1]) \cap (a_0, b_0) = \emptyset$, in particular, $\mathcal{D}_g \cap (a_0, b_0) = \emptyset$. 
This contradicts the density of $\mathcal{D}_g$ in $[0,1]$. 
The same argument is applied to the case where $x_0 \in \{0,1\}$.

Second we show the uniqueness. 

{\it Step 1.} We show that if $\varphi$ is a solution of \eqref{eq:gen-dR-def-1}, then $\varphi(j) = j, j \in \{0,1\}$. 
By \eqref{eq:gen-dR-def-1} and the assumption that $f_0 (0) = 0$, 
$\varphi(0) = \varphi(f_0 (0)) = g_0 (\varphi(0))$. 
Denote the $n$-fold composition of $g_0$ by $g_0^{n}$. 
Then, 
by induction in $n$, 
we see that 
$\varphi(0) = g_0^{n}(\varphi(0))$ for every $n \ge 1$. 
Since $g_0$ is an increasing function on $[0,1]$, 
$ g_0^{n}$ is also an increasing  function on $[0,1]$ for each $n$. 
Therefore, 
$0 = g_0^{n}(0) \le g_0^{n}(\varphi(0)) \le g_0^{n}(1)$.   
By this and \eqref{eq:density-arbitrarily-close}, 
\[ \varphi(0)  = \lim_{n \to \infty} g_0^{n}(\varphi(0)) = \lim_{n \to \infty} g_0^{n}(1) = 0. \]
In the same manner, we can show that $\varphi(1) = 1$.

{\it Step 2.} Since $\varphi(j) = j, j \in \{0,1\}$, we see that 
\begin{equation}\label{eq:solution-dyadic}
\varphi\left( f_{i_1} \circ \cdots \circ f_{i_n}(j) \right) = g_{i_1} \circ \cdots \circ g_{i_n}(j), \ i_1, \ldots, i_n \in \mathcal{I}_N, \ j \in \{0,1\}, 
\end{equation} 
by induction in $n$. 

Since $\mathcal{D}_f$ is dense in $[0,1]$ and $\varphi$ is assumed to be continuous, we obtain the uniqueness. 
\end{proof}

We give an example of a compatible system $\{h_0, h_1\}$ which is not a $D$\nobreakdash-system. 
This is suggested by an anonymous referee. 
\begin{Exa}\label{exa:fixed-point-not-D-system} 
Let
\[ h_0(x) \coloneqq
\begin{cases}
\frac{1}{2}x & 0 \le x \le \frac{1}{4},\\
2x - \frac{3}{8} & \frac{1}{4} \le x \le \frac{1}{2},\\
\frac{3}{8}x + \frac{7}{16} & \frac{1}{2} \le x \le 1,
\end{cases} \]
and
\[ h_1(x) \coloneqq \frac{3}{16}x + \frac{13}{16}, \qquad 0 \le x \le 1. \]

Then the functions $h_0$ and $h_1$ are strictly increasing and continuous.
Moreover,
\[ 0 = h_0(0) < h_0(1) = \frac{13}{16} = h_1(0) < h_1(1) = 1. \]
Hence, $\{h_0, h_1\}$ is a compatible system.

We also observe that
\[ h_0(x) \ge \frac{3}{8} \quad \text{for every } x \in \left[\frac{3}{8}, 1\right]. \]
Therefore,
\[ \lim_{n \to \infty} h_0^n(1) \ge \frac{3}{8}. \]
Since
$\displaystyle \inf \mathcal{D}_h \setminus \{0\} = \lim_{n \to \infty} h_0^n(1)$, 
it follows that $\mathcal{D}_h \cap (0, \frac{3}{8}) = \emptyset$.
Consequently, $\mathcal{D}_h$ is \emph{not} dense in $[0,1]$.

Finally, we remark that $h_0$ has a fixed point other than zero; indeed,
$\displaystyle h_0\!\left(\frac{3}{8}\right) = \frac{3}{8}$.
\end{Exa}

It is natural to consider the inverse function of the solution $\varphi$. 
We easily see the following assertion. 
\begin{Lem}\label{lem:inverse-exchange}
Assume that $\{f_i\}_{i \in \mathcal{I}_N}$ and $\{g_i\}_{i \in \mathcal{I}_N}$ are D-systems on $[0,1]$. 
Let $\varphi$ be the unique solution of \eqref{eq:gen-dR-def-1}. 
Then, the inverse function $\varphi^{-1}$ of $\varphi$ is well-defined and satisfies the following conjugate equation 
\begin{equation*}
\varphi^{-1} \circ g_i = f_i \circ \varphi^{-1}, \ i \in \mathcal{I}_N, 
\end{equation*}
which is equivalent with \eqref{eq:gen-dR-def-1} where $\{f_i\}_{i \in \mathcal{I}_N}$ and $\{g_i\}_{i \in \mathcal{I}_N}$ are exchanged. 
\end{Lem}

By \cite[Proposition 3.1]{Berg2000}, if $\varphi$ is singular, then $\varphi^{-1}$ is also singular.

\subsection{Weak contraction}\label{sec:wkc}

Hereafter, we assume additional conditions for the D-systems $\{f_i\}_{i \in \mathcal{I}_N}$ and $\{g_i\}_{i \in \mathcal{I}_N}$. 
We say that a function $h \colon [0,1] \to [0,1]$ is a {\it contraction} on $[0,1]$ if 
\[ \|h\|_{\textup{Lip}} \coloneqq \sup_{x \ne y} \left| \frac{h(x)-h(y)}{x-y} \right| < 1. \] 
$\|h\|_{\textup{Lip}}$ is called the Lipschitz constant of $h$. 

For a function $h \colon [0,1] \to \mathbb{R}$, let  
\begin{equation}\label{eq:modulus-of-continuity} 
\omega_h (s) \coloneqq \sup\left\{|h(x)-h(y)| \colon |x-y| \le s \right\}. 
\end{equation} 

We say that a function $h \colon [0,1] \to [0,1]$ is a {\it weak contraction} on $[0,1]$ if 
\[ \phi_{h} (t) \coloneqq \lim_{s \to t+0} \omega_h (s)  < t,  \  \ t > 0. \]

\begin{Lem}\label{lem:wkc-1}
Let $h$ be a weak contraction. 
Then, 
$\displaystyle \lim_{n \to \infty} \phi_h^n (t) = 0$ for every $t > 0$ and 
$|h(x)-h(y)| \le \phi_h (|x-y|)$ for $x,y \in [0,1]$. 
\end{Lem}

\begin{proof}
Since $\omega_h$ is increasing,  
$\phi_h$ is increasing and right-continuous on $[0,1]$. 
Hence, 
\[ |h(x)-h(y)| \le \lim_{s \to |x-y|+0} \phi_h (s) = \phi_h (|x-y|). \]  

By the assumption, for every $t > 0$, $(\phi_h^n (t))_n$ is strictly decreasing.  
Hence, the limit $\displaystyle\phi^{\infty}_h (t) \coloneqq \lim_{n \to \infty} \phi_h^n (t)$ exists. 
By the right-continuity of $\phi_h$, 
we see that 
$\displaystyle \phi_h \left(\phi^{\infty}_h (t)\right) = \phi^{\infty}_h (t)$. 
Hence, $\phi^{\infty}_h (t) = 0$. 
\end{proof}

\begin{Lem}\label{lem:wkc-dense}
Let $\{h_i\}_{i \in \mathcal{I}_N}$ be a compatible system on $[0,1]$ such that each $h_i$ is a weak contraction on $[0,1]$. 
Then, $\{h_i\}_{i \in \mathcal{I}_N}$ is a D-system. 
\end{Lem}

\begin{proof}
Let $\displaystyle \phi \coloneqq \max_{i \in \mathcal{I}_N} \phi_{h_i}$. 
Since each $\phi_{h_i}$ is right-continuous, 
$\phi$ is right-continuous\footnote{The assumption that $N$ is finite is important.} on $[0,1]$, $\phi(t) < t$ for every $t > 0$, 
and $\displaystyle \max_{i \in \mathcal{I}_N} |h_i (x) - h_i(y)| \le \phi(|x-y|)$, $x, y \in [0,1]$. 
Then, we see that $\displaystyle \lim_{n \to \infty} \phi^n (t) = 0$ for every $t > 0$. 
Hence, 
\[ \left| h_{i_1} \circ \cdots \circ h_{i_n}(1) - h_{i_1} \circ \cdots \circ h_{i_n}(0) \right| \le \phi^n (1). \]
\[ \lim_{n \to \infty} \max_{i_1, \ldots, i_n \in \mathcal{I}_N} \left|h_{i_1} \circ \cdots \circ h_{i_n}(1) - h_{i_1} \circ \cdots \circ h_{i_n}(0)\right| = 0. \]
By the definition, for each $n \ge 1$, 
\[ [0,1] = \bigcup_{i_1, \ldots, i_n \in \mathcal{I}_N} \left[ h_{i_1} \circ \cdots \circ h_{i_n}(0), h_{i_1} \circ \cdots \circ h_{i_n}(1)\right], \]
and their interiors are disjoint. 
\end{proof}

We say that  a compatible system $\{h_i\}_{i \in \mathcal{I}_N}$ on $[0,1]$ is a {\it (weakly) contractive system} on $[0,1]$ if each $h_i$ is a (weak) contraction on $[0,1]$. 
Every weakly contractive system is a D-system. 

\subsection{Linear fractional transforms}\label{sec:lft}

\begin{Def}\label{def:lf-family}
We say that a family of functions $\{h_i\}_{i \in \mathcal{I}_N}$ is an {\it LF system} on $[0,1]$ if 
$h_i (x) \coloneqq \dfrac{a_i x + b_i}{c_i x + d_i}, x \in [0,1]$, for each $i \in \mathcal{I}_N$, 
where $a_i, b_i, c_i, d_i, i \in \mathcal{I}_N$, are real numbers such that \\
(1) (well-defined) It holds that $c_i +d_i > 0$ and  $d_i > 0$ for $i \in \mathcal{I}_N$.\\ 
(2) (strictly increasing, weak contraction) It holds that 
$$ 0 < a_i d_i - b_i c_i \le \min\{d_i^2, (c_i+d_i)^2\}, \ i \in \mathcal{I}_N.$$ 
(3) (compatibility) It holds that $b_0 = 0$ and 
$\dfrac{a_i + b_i}{c_i + d_i} = \dfrac{b_{i+1}}{d_{i+1}}$ for $i \in \mathcal{I}_N$, where we let $\dfrac{b_N}{d_{N}} \coloneqq 1$. \\ 
(4) The sequence $\left(\dfrac{b_i}{d_i}\right)_i$ is strictly increasing in $i$. 
\end{Def}

Condition (4) follows from the conditions (1), (2) and (3) but is frequently used.  
For condition (2), we may consider a stronger condition: \\ 
(2') $0 < a_i d_i - b_i c_i < \min\{d_i^2, (d_i+c_i)^2\}, \ i \in \mathcal{I}_N$. 

Since $h_i$ is invariant if we multiply each of $a_i, b_i, c_i$ and $d_i$ by a positive constant, 
we may assume that $d_i = 1$ or $a_i d_i - b_i c_i = 1$. 
An LF system is a compatible system. 

\begin{Lem}\label{lem:lf-dense}
Every LF system is a weakly contractive system. 
If (2') holds instead of (2), then it is a contractive system. 
\end{Lem}

\begin{proof}
Let $\{h_i\}_{i \in \mathcal{I}_N}$ be an LF-system. 
We assume that $d_i = 1$ for every $i \in \mathcal{I}_N$. 
By Lemma \ref{lem:wkc-dense}, it suffices to show that each $h_i$ is a weak contraction. 
We first remark that 
\begin{equation}\label{eq:difference-LF} 
|h_i (x) - h_i (y)| = \frac{a_i - b_i c_i}{(c_i x + 1)(c_i y + 1)} |x-y|. 
\end{equation} 
We consider three cases: (i) $c_i > 0$, (ii) $c_i < 0$,  and (iii) $c_i = 0$. 

(i) Assume $c_i > 0$. 

We remark that  
$a_i - b_i c_i \le 1$ and 
\[ |h_i (x) - h_i (y)|  \le \frac{a_i - b_i c_i}{c_i |x-y| + 1} |x-y| \le \frac{|x-y|}{c_i |x-y| + 1}.\]  
Hence, by \eqref{eq:modulus-of-continuity}, 
\[ \omega_{h_i}(s) \le \frac{s}{c_i s + 1}, \ \ s \ge 0. \]
By this and $c_i > 0$, 
\[ \phi_{h_i}(t) \le \frac{t}{c_i t + 1} < t, \ t > 0. \]

(ii) Assume $c_i < 0$. 
Then, $a_i - b_i c_i \le (c_i + 1)^2$ and 
\[ |h_i (x) - h_i (y)|  \le \frac{a_i - b_i c_i}{(c_i+1)(c_i (1-|x-y|) + 1)} |x-y| \le \frac{(1+c_i) |x-y|}{-c_i |x-y| + 1+c_i}.\]  
Hence, by \eqref{eq:modulus-of-continuity}, 
\[ \omega_{h_i}(s) \le \frac{(1+c_i) s}{-c_i s + 1+c_i}, \ \ s \ge 0. \]
By this and $c_i < 0$, 
\[ \phi_{h_i}(t) \le \frac{(1+c_i) t}{-c_i t + 1+c_i} < t, \ t > 0. \]

(iii) Assume $c_i = 0$. 

We remark that 
$0 < a_i \le 1$ and $|h_i (x) - h_i (y)| = a_i |x-y|$. 
By the assumption for $a_i, b_i, c_i$, 
it holds that $h_j (1) > h_j (0)$ for every $j$. 
By the fact that $N \ge 2$, 
we obtain that $a_i = h_i (1) - h_i (0) < 1$. 
\end{proof}

\begin{Rem}
The technique used above is related to a characterization of Rakotch contractions. 
See \cite[Lemma 2.10 (iii)]{Lesniak2020}. 
\end{Rem}

Let $\displaystyle (\mathcal{I}_N)^* \coloneqq \bigcup_{n \in \mathbb{N}} (\mathcal{I}_N)^n$. 
The length of $\sigma$ is the unique number $n \in \mathbb{N}$ such that $\sigma \in (\mathcal{I}_N)^n$. 
We denote the length of $\sigma \in (\mathcal{I}_N)^*$ by $|\sigma|$. 
For an LF system $\{h_i\}_{i \in \mathcal{I}_N}$  and $\sigma = (\sigma_1, \ldots, \sigma_n) \in (\mathcal{I}_N)^*$,  
we let 
\[ h_{\sigma}(x) \coloneqq h_{\sigma_1} \circ \cdots \circ h_{\sigma_n} (x), \ x \in [0,1].  \]

\begin{Lem}\label{lem:LF-lower}
Let $\{h_i\}_{i \in \mathcal{I}_N}$ be an LF system. 
Then, there exists $c \in (0,1)$ such that for every $\sigma \in (\mathcal{I}_N)^*$, 
$$ h_{\sigma}(1) - h_{\sigma}(0) \ge c^{|\sigma|}. $$
\end{Lem}

\begin{proof}
Let $A_i \coloneqq  \begin{pmatrix} a_i & b_i \\ c_i & d_i \end{pmatrix}$, $i \in \mathcal{I}_N$. 
We assume that $a_i d_i - b_i c_i = 1$ for each $i \in \mathcal{I}_N$. 
Let $(i_n)_n \in (\mathcal{I}_N)^{\mathbb{N}}$ and 
\[ \begin{pmatrix} p_n & q_n \\ r_n & s_n \end{pmatrix} = A_{i_1} \cdots A_{i_n}, \ n \ge 1. \]
Let $C_1 \coloneqq \max\{|a_i|, |b_i|, |c_i|, |d_i| \colon i \in \mathcal{I}_N\}$. 
Then, for every $n$, 
$|r_{n+1}| + |s_{n+1}| \le 2C_1 (|r_n| + |s_n|)$. 
Hence, 
$|r_n| + |s_n| \le (2C_1)^n$  for every $n$. 

Assume that $\sigma = (i_1, \ldots, i_n)$. 
Then, $|\sigma| = n$ and by using the equality $p_n s_n - q_n r_n = 1$, 
\[ h_{\sigma}(1) - h_{\sigma}(0) = \frac{p_n + q_n}{r_n + s_n} - \frac{q_n}{s_n} \]
\[ = \frac{1}{s_n (r_n + s_n)} \ge \frac{1}{(|r_n| + |s_n|)^2} \ge \frac{1}{4^n |C_1|^{2n}}. \] 
\end{proof}

We give an upper bound for $h_{\sigma}(1) - h_{\sigma}(0)$. 

\begin{Lem}\label{lem:LF-upper}
Let $\{h_i\}_{i \in \mathcal{I}_N}$ be an LF system. 
Then, there exists $C > 0$ such that for every $\sigma \in (\mathcal{I}_N)^*$, 
$$ h_{\sigma}(1) - h_{\sigma}(0) \le \frac{C}{|\sigma|}. $$
\end{Lem}

\begin{proof}
In this proof, we assume that $a_i d_i - b_i c_i = 1$ for each $i \in \mathcal{I}_N$. 
We first show that 
\begin{equation}\label{eq:denominator-lower-bound}
(c_i x + d_i)(c_i y + d_i) \ge 1 + |c_i| |x-y|, \ x, y \in [0,1]. 
\end{equation}
We assume that $x \le y$. 
Assume $c_i > 0$. 
Then, by Definition \ref{def:lf-family} (2), 
\[ c_i y + d_i \ge c_i x + d_i \ge \min_{z \in [0,1]} c_i z + d_i  = d_i \ge 1.\] 
Hence, 
\[ (c_i x + d_i)(c_i y + d_i) = (c_i x + d_i)^2 + c_i (y-x) (c_i x + d_i) \ge 1+ c_i (y-x). \]
Assume $c_i < 0$. 
Then, by Definition \ref{def:lf-family} (2) and (1), 
\[ c_i x + d_i \ge c_i y + d_i \ge \min_{z \in [0,1]} c_i z + d_i =  c_i + d_i = |c_i + d_i| \ge 1.\] 
Hence, 
\[ (c_i x + d_i)(c_i y + d_i) = (c_i y + d_i)^2 - c_i (y-x) (c_i y + d_i) \ge 1 - c_i (y-x). \] 
If $c_i = 0$, then $(c_i x + d_i)(c_i y + d_i) = d_i^2 \ge 1 = 1 + |c_i| |x-y|$. 
Thus we see \eqref{eq:denominator-lower-bound}. 

By \eqref{eq:difference-LF} and \eqref{eq:denominator-lower-bound}, 
\[ |h_i (x) - h_i (y)| \le \frac{|x-y|}{1+|c_i| |x-y|}. \]
If $c_i = 0$, then 
\[ |h_i (x) - h_i (y)| = \frac{a_i}{d_i} |x-y|. \] 
Since $h_i (1) - h_i (0) < 1$, 
it holds that 
$a_i < d_i$. 
Let 
$C_1 \coloneqq \min\{|c_i| \colon c_i \ne 0\}$ and $C_2 \coloneqq \max\left\{\dfrac{a_i}{d_i} \colon c_i = 0 \right\}$. 
Then, $C_1 > 0$, $C_2 < 1$ and 
\[ \max_{i \in \mathcal{I}_N} |h_i (x) - h_i (y)| \le \max\left\{ \frac{|x-y|}{1+ C_1 |x-y|}, C_2 |x-y| \right\}. \]
We can take $C_3 \in (0, C_1)$ such that 
\[ \max_{i \in \mathcal{I}_N}  |h_i (x) - h_i (y)| \le \frac{|x-y|}{1+ C_3 |x-y|}, \ x, y \in [0,1]. \] 

Let $A_n \coloneqq \max\left\{ h_{\sigma}(1) - h_{\sigma}(0) \colon \sigma \in (\mathcal{I}_N)^*, |\sigma| = n \right\}$.
Then, $A_1 \le \dfrac{1}{1+C_3}$ and 
$A_{n+1} \le \dfrac{A_n}{1+C_3 A_n}, n \ge 1$. 
Hence, 
\[ A_n \le \frac{1}{nC_3 + 1} \le \frac{1}{nC_3}. \qedhere\]
\end{proof}

\begin{Prop}[Weak version of regularity]\label{prop:main-weak-regularity}
Assume that $\{f_i\}_{i \in \mathcal{I}_N}$ and $\{g_i\}_{i \in \mathcal{I}_N}$ are LF systems and let $\varphi$ be the solution of \eqref{eq:gen-dR-def-1}.   
Then, there exists a constant $C > 0$ such that for every $y \in (0,1)$ and $r > 0$ satisfying that $0 \le y-r < y+r \le 1$, 
\[ \mu_{\varphi}\left([y-r, y+r]\right) \ge \exp(-C/r). \]
\end{Prop}

This might be regarded as a weak version of the regularity of $\mu_{\varphi}$ in the sense of Ullman--Stahl--Totik. 
See \cite[Criterion 1.1]{Mantica2019}.
However, the following proof does not yield the regularity.  

\begin{proof}
We apply Lemma \ref{lem:LF-lower} to $\{g_i\}_{i \in \mathcal{I}_N}$ and Lemma \ref{lem:LF-upper} to $\{f_i\}_{i \in \mathcal{I}_N}$.  
Let $C$ be the constant appearing in Lemma \ref{lem:LF-upper}. 
Take $y \in (0,1)$ and $r > 0$ such that $0 \le y-r < y+r \le 1$ arbitrarily. 
Let $n$ be a natural number such that $\dfrac{C}{n} < r \le \dfrac{C}{n-1}$. 
We let $\dfrac{C}{n-1} \coloneqq \infty$ for $n=1$. 
Then, by the compatibility condition and Lemma \ref{lem:LF-upper},  
there exist $i_1, \ldots, i_n \in \mathcal{I}_N$ such that 
$$ \left[ f_{i_1} \circ \cdots \circ f_{i_n}(0), f_{i_1} \circ  \cdots \circ f_{i_n}(1) \right] \subset [y-r, y+r].$$
Hence, by \eqref{eq:solution-dyadic}, 
\[ \mu_{\varphi}\left([y-r, y+r]\right) \ge \mu_{\varphi}\left(    \left[ f_{i_1} \circ \cdots \circ f_{i_n}(0), f_{i_1} \circ  \cdots \circ f_{i_n}(1) \right]   \right) \]
\[ = \varphi(f_{i_1} \circ  \cdots \circ f_{i_n}(1)) - \varphi(f_{i_1} \circ  \cdots \circ f_{i_n}(0)) \]
\[ = g_{i_1} \circ  \cdots \circ g_{i_n}(1) - g_{i_1} \circ  \cdots \circ g_{i_n}(0). \]

By Lemma \ref{lem:LF-lower}, there exists a constant $c \in (0,1)$ depending only on $\{g_i\}_i$ such that 
\[ g_{i_1} \circ  \cdots \circ g_{i_n}(1) - g_{i_1} \circ  \cdots \circ g_{i_n}(0) \ge c^{n}. \]
By the assumption for $n$, we obtain that $c^{n} \ge c^{1+C/r}$. 
Thus, we obtain that 
\[ \mu_{\varphi}([y-r, y+r]) \ge c^{1+C/r}. \qedhere\]
\end{proof}

\section{Singularity}\label{sec:singu-non-linear}

In this section, we consider \eqref{eq:gen-dR-def-1} under the setting that $\{g_i\}_{i \in \mathcal{I}_N}$ is an LF-system on $[0,1]$, and, $\{f_i\}_{i \in \mathcal{I}_N}$ is  a weakly contractive D-system on $[0,1]$. 
By Lemma \ref{lem:wkc-dense} and Lemma \ref{lem:lf-dense}, 
both $\{f_i\}_{i \in \mathcal{I}_N}$ and $\{g_i\}_{i \in \mathcal{I}_N}$ are D-systems. 
Hence, by Lemma \ref{lem:exist-unique-gen-dR}, there exists a unique continuous strictly increasing solution $\varphi$ of \eqref{eq:gen-dR-def-1}. 

We consider the singularity for the solution $\varphi$. 
We will give sufficient conditions for $\dim_H \mu_{\varphi} < 1$. 
It is known that if $\dim_H \mu_{\varphi} < 1$, then $\mu_{\varphi}$ is singular with respect to the Lebesgue measure. 
By \cite[Definition 9.2.2 and Theorem 9.2.4]{Rana2002}, 
this implies that $\varphi$ is a singular function on $[0,1]$, specifically, $\varphi$ is a strictly increasing and continuous function on $[0,1]$ such that its derivative is equal to zero almost everywhere. 

The main purpose of this section is to extend \cite[Theorem 43]{Okamura2020} to the case where $f_i$ can be non-linear. 
In \cite{Okamura2020}, the author dealt with many objects other than the conjugate functional equation, so the arguments were not specialized for the functional equation. 
We introduce some new notations and the formulations of our main results and some of the arguments in the proofs are different from \cite[Theorem 43]{Okamura2020}. 
The author believes that these changes will make the arguments more transparent. 

We assume that an LF-system $\{g_i\}_{i \in \mathcal{I}_N}$ is defined by 
$g_i (x) \coloneqq \dfrac{a_i x + b_i}{c_i x + d_i}, x \in [0,1], (i \in \mathcal{I}_N)$, 
where $a_i, b_i, c_i, d_i, (i \in \mathcal{I}_N)$, are real numbers satisfying the conditions (1), (2) and (3) in Definition \ref{def:lf-family}, specifically,  \\ 
(1)  It holds that $c_i +d_i > 0$ and  $d_i > 0$ for $i \in \mathcal{I}_N$.\\ 
(2) It holds that 
$$ 0 < a_i d_i - b_i c_i \le \min\{d_i^2, (c_i+d_i)^2\}, \ i \in \mathcal{I}_N.$$ 
(3) It holds that $b_0 = 0$ and 
$\dfrac{a_i + b_i}{c_i + d_i} = \dfrac{b_{i+1}}{d_{i+1}}$ for $i \in \mathcal{I}_N$, where we let $\dfrac{b_N}{d_{N}} \coloneqq 1$. 

{\it In this section, we assume $d_i = 1$ for each $i \in \mathcal{I}_N$. } 

We show that $(1-a_i)^2 + 4b_i c_i \ge 0$. 
If $c_i \ge 0$, then this is true since $b_i \ge 0$. 
If $c_i < 0$, then by the compatibility condition, $a_i + b_i \le c_i + 1$. 
Hence, 
$0 \le b_i - c_i \le 1- a_i$, and then $(1-a_i)^2 \ge (b_i - c_i)^2$. 
Hence, $(1-a_i)^2 + 4b_i c_i  \ge (b_i + c_i)^2 \ge 0$.  

Let 
\[ \alpha \coloneqq \min\left\{0,  \frac{c_0}{1-a_0}, \frac{a_i -1 + \sqrt{(1-a_i)^2 + 4b_i c_i}}{2b_i} \, \middle| \, 1 \le i \le N-1 \right\},\]
where we {\it delete}\footnote{This is different from \cite{Okamura2020}.} $\dfrac{c_0}{1-a_0}$ if $a_0 = 1$. 
Let 
\[ \beta \coloneqq \max\left\{0,  \frac{c_0}{1-a_0}, \frac{a_i -1 + \sqrt{(1-a_i)^2 + 4b_i c_i}}{2b_i} \, \middle| \, 1 \le i \le N-1 \right\}, \]
where $\dfrac{c_0}{1-a_0} \coloneqq +\infty$ if $a_0 = 1$. 

Let $Y \coloneqq [\alpha, \beta]$ and 
\[ G_i (y) \coloneqq \frac{(a_i - b_i c_i)(y+1)}{(b_i y +1)((a_i + b_i)y + (c_i + 1))} = \frac{(b_{i+1}-b_i)(y+1)}{(b_{i+1}y+1)(b_i y + 1)},  \ y \in Y, \]
as well as 
\[ H_i (y) \coloneqq \frac{a_i y + c_i}{b_i y + 1}, \ y \in Y, \]
where we let $b_N \coloneqq 1$. 
We remark that $G_0 (y) = \dfrac{b_1 (y+1)}{b_1 y + 1}$ and $G_{N-1}(y) = \dfrac{1-b_{N-1}}{b_{N-1}y + 1}$. 

Let 
\[ G_0 (+\infty) \coloneqq 1, G_i (+\infty) \coloneqq 0, H_0 (+\infty) \coloneqq +\infty, H_i (+\infty) \coloneqq \frac{a_i}{b_i}, 1 \le i \le N-1.   \]
Then, 
\[ G_i (+\infty) = \lim_{y \to +\infty} G_i (y),  \ H_i (+\infty) = \lim_{y \to +\infty} H_i (y), \ i \in \mathcal{I}_N. \]

The following is the same as \cite[Lemma 27]{Okamura2020}, so we will not repeat the proof. 
\begin{Lem}\label{lem:basic-GH}
We have the following: \\ 
(i) $0 \in Y \subset [-1, +\infty]$. \\
(ii) Each $H_i$ is well-defined on $Y$. \\
(iii) $\alpha \le H_i (\alpha) \le H_i (\beta) \le \beta$, that is, $H_i (Y) \subset Y$. \\
(iv) $\displaystyle \sum_{i \in \mathcal{I}_{N}} G_i (y) = 1, \  y \in Y.$ 
\end{Lem}

By (iii), we can compose $H_i$ freely on $Y$.  
We consider the Alexandroff compactification when $\beta = +\infty$. 
Then, $Y = [\alpha, \beta]$ is compact, and $G_i \colon Y \to [0,1]$ and $H_i \colon Y \to Y$ are continuous. 
We remark that $Y, G_i, H_i$ are defined by $(g_i)_i$, not by $(f_i)_i$. 

We denote  the set of fixed points of $H_i$ on $Y$ by $\textup{Fix}(H_i)$. 

The set $\textup{Fix}(H_i)$ is a singleton and 
\[ \textup{Fix}(H_0) = \left\{\dfrac{c_0}{1-a_0} \right\}, \ \textup{Fix}(H_i) = \left\{ \dfrac{a_i -1 + \sqrt{(1-a_i)^2 + 4b_i c_i}}{2b_i} \right\}, \ i \ge 1.\]

If $a_0 = 1$, then $c_0 > 0$ and hence $ \textup{Fix}(H_0) = \{+\infty\}$. 
If $a_0 = 1$ and $c_0 = 0$, then, by using the assumption that $d_0 = 1$, we obtain $g_0 (x) = x$ and hence $N=1$. 
This contradicts the assumption that $N \ge 2$. 

Let 
\begin{align*} 
\mathcal{P}_N &\coloneqq \left\{(p_i)_{i \in \mathcal{I}_N} \colon p_i > 0, \sum_{i} p_i = 1 \right\} \ \textup{ and}  \\ 
\overline{\mathcal{P}_N} &\coloneqq \left\{(p_i)_{i \in \mathcal{I}_N} \colon p_i \ge 0, \sum_{i} p_i = 1 \right\}.
\end{align*}

Let
\begin{align*} 
{\bf G}(y) &\coloneqq (G_0 (y), \dots, G_{N-1}(y)) \in \overline{\mathcal{P}_N}, \  y \in Y, \\
\mathfrak{p}_i &\coloneqq {\bf G} \left(\textup{Fix}(H_i)\right) \in \overline{\mathcal{P}_N}, \ \ \ \ \ \ \ \ \ \ \ \ \ \ i \in \mathcal{I}_N, \ \ \ \ \ \ \ \ \ \textup{ and,} \\ 
\dim_H \mu_{\varphi} &\coloneqq \inf\left\{\dim_H A \colon A \in \mathcal{B}([0,1]), \mu_{\varphi} (A) = 1 \right\}, 
\end{align*}
where $\mathcal{B}([0,1])$ is the Borel $\sigma$-algebra on $[0,1]$. 
The following results extend \cite[Theorem 43 (i)]{Okamura2020} to the case where $\{f_i\}_{i \in \mathcal{I}_N}$ may be non-affine maps.   

\begin{Thm}\label{thm:main-singular-1}
If $\left|\{\mathfrak{p}_i\}_{i \in \mathcal{I}_N} \right| \ge 2$, 
then for every $\epsilon \in (0,1/2)$, 
there exists $\delta > 0$ such that it holds that $\dim_H \mu_{\varphi} < 1$ 
for every  contractive system $\{f_i\}_{i \in \mathcal{I}_N}$ on $[0,1]$ such that  
$\displaystyle \sum_{i \in \mathcal{I}_N} r_i  < 1+ \delta$ and $\epsilon \le r_i \le 1-\epsilon$ for each $i \in \mathcal{I}_N$, 
where we let $r_i \coloneqq \|f_i\|_{\textup{Lip}}$. 
\end{Thm}

The probability vectors $\mathfrak{p}_i, i \in \mathcal{I}_N$, are determined only by $(\alpha, \beta, \{G_i\}_i, \{H_i\}_i)$ and depend only on $\{g_i\}_{i \in \mathcal{I}_N}$. 
The system $\{f_i\}_{i \in \mathcal{I}_N}$ is {\it not} related to this. 
This contains the case where $f_i (x) = a_i x + (a_0 + \cdots + a_{i-1})$ for each $i \in \mathcal{I}_N$.

We also deal with the case where $\left|\{\mathfrak{p}_i\}_{i \in \mathcal{I}_N} \right| = 1$. 
Let the total variation distance between probability vectors be 
\[ \left\| {\bf p} - {\bf q}  \right\|_1 \coloneqq \sum_{i \in \mathcal{I}_N} |p_i - q_i|, \ \ {\bf p} = (p_i)_i, {\bf q} = (q_i)_i \in \overline{\mathcal{P}_N}. \] 

\begin{Thm}\label{thm:main-singular-2}
If $\left|\{\mathfrak{p}_i\}_{i \in \mathcal{I}_N} \right| = 1$, 
then for every $\epsilon \in (0,1/2)$, 
there exists $\delta > 0$ such that $\dim_H \mu_{\varphi} < 1$ 
for every  contractive system $\{f_i\}_{i \in \mathcal{I}_N}$ on $[0,1]$ such that  
$\displaystyle \sum_{i \in \mathcal{I}_N} r_i < 1+ \delta$, $\epsilon \le r_i \le 1-\epsilon$ for each $i \in \mathcal{I}_N$,  and $\|{\bf r} - \mathfrak{p}_0 \|_1 \ge \epsilon$, 
where we let $r_i \coloneqq \|f_i\|_{\textup{Lip}}$ and ${\bf r} \coloneqq (r_0, \ldots, r_{N-1})$. 
\end{Thm}

We remark that $\left|\{\mathfrak{p}_i\}_{i \in \mathcal{I}_N} \right| = 1$ if and only if $\mathfrak{p}_0 = \cdots = \mathfrak{p}_{N-1}$. 
If $f_i (x) = a_i x + (a_0 + \cdots + a_{i-1})$ for each $i \in \mathcal{I}_N$, and $\mathfrak{p}_0  = {\bf q}$, then 
the solution $\varphi$ is smooth on $[0,1]$ and we also have an explicit expression for $\varphi$. 
See \cite[Theorem 43 (ii)]{Okamura2020}. 

\subsection{Proofs}

We first give notations for symbolic dynamics. 
Let $\mathbb{N} \coloneqq \{1,2, \ldots\}$. 
We define a map $\pi \colon (\mathcal{I}_N)^{\mathbb N} \to [0,1]$ by 
\[ \pi (i_1 i_2 \ldots) \coloneqq \lim_{k \to \infty} f_{i_1} \circ f_{i_2} \circ \cdots \circ f_{i_k} (0).  \] 
This is well-defined because each $f_i$ is a contraction and  $\left(f_{i_1} \circ f_{i_2} \circ \cdots \circ f_{i_k} (0)\right)_k$ is a Cauchy sequence 
since \[ \left|f_{i_1} \circ f_{i_2} \circ \cdots \circ f_{i_k} \circ f_{i_{k+1}} (0) - f_{i_1} \circ f_{i_2} \circ \cdots \circ f_{i_k} (0)\right| \le \left( \max_{i \in \mathcal{I}_N} \|f_i \|_{\textup{Lip}} \right)^{k}.\] 

Let $\xi_j \colon  (\mathcal{I}_N)^{\mathbb N} \to \mathcal{I}_N$ be the projection to the $j$-th coordinate. 
Let 
\[ I(i_1, \ldots, i_n) \coloneqq \left\{w \in  (\mathcal{I}_N)^{\mathbb N} \colon \xi_j (w) = i_j, \ 1 \le j \le n  \right\}. \] 
Then, by Lemma \ref{lem:density-consequence}, 
\[ \pi\left(I(i_1, \ldots, i_k)\right) = \left[f_{i_1} \circ f_{i_2} \circ \cdots \circ f_{i_k} (0), f_{i_1} \circ f_{i_2} \circ \cdots \circ f_{i_k} (1) \right]. \] 
Hence, by \eqref{eq:solution-dyadic}, 
\begin{align*}
\mu_{\varphi}\left(\pi(I(i_1, \ldots, i_k))\right) &= \varphi\left(f_{i_1}  \circ \cdots \circ f_{i_k} (1) \right) - \varphi \left(f_{i_1} \circ \cdots \circ f_{i_k} (0) \right) \notag\\
&= g_{i_1} \circ \cdots \circ g_{i_k} (1) - g_{i_1} \circ \cdots \circ g_{i_k} (0).  
\end{align*}

For $y \in Y$,  
let $\nu_y$ be a probability measure on $(\mathcal{I}_N)^{\mathbb N}$ such that 
$\nu_y (I(i)) = G_i (y)$, $i \in \mathcal{I}_N$, and, 
\begin{equation}\label{eq:def-nu} 
\frac{\nu_y (I(i_1, \ldots, i_{n-1}, i_n))}{\nu_y (I(i_1, \ldots, i_{n-1}))}  = G_{i_n} \circ H_{i_{n-1}} \circ \cdots \circ H_{i_1} (y), \ n \ge 2.  
\end{equation} 
Since $\displaystyle \sum_{i \in \mathcal{I}_N} G_i (y) = 1$ for each $y \in Y$, the Kolmogorov consistency condition holds. 
Therefore, a probability measure on $(\mathcal{I}_N)^{\mathbb N}$ satisfying \eqref{eq:def-nu} exists uniquely by the Kolmogorov extension theorem \cite[Theorem 2.4.3]{Tao2011}. 

For $i \in \mathcal{I}_N$, 
let $F_i \colon (\mathcal{I}_N)^{\mathbb{N}}  \to (\mathcal{I}_N)^{\mathbb{N}}$ be a map defined by $F_i (z) = iz$. 
Now  we see that 
\begin{equation}\label{eq:IFS-nu}
\nu_y = \sum_{i \in \mathcal{I}_N}  G_i (y) \nu_{H_i (y)} \circ F_i^{-1}. 
\end{equation}

Let $\mu_y \coloneqq \nu_y \circ \pi^{-1}$. 
Then, $\mu_y$ is a Borel probability measure on $[0,1]$. 
By  $\pi \circ F_i = f_i \circ \pi$ and   \eqref{eq:IFS-nu}, 
\begin{equation}\label{eq:IFS-mu}
\mu_y = \sum_{i \in \mathcal{I}_N}  G_i (y) \mu_{H_i (y)} \circ f_i^{-1}. 
\end{equation}

We recall that $0 \in Y = [\alpha, \beta]$. 
Then, by \eqref{eq:IFS-mu}, 
\[ \mu_0 = \sum_{i \in \mathcal{I}_N}  G_i (0) \mu_{H_i (0)} \circ f_i^{-1}. \]
Since $H_i (0) \ne 0$ can occur, this does not define an iterated function system. 
The equation \eqref{eq:IFS-mu} is in the framework of a place-dependent iterated function system \cite{Fan1999}. 

By \cite[Lemma 30]{Okamura2020}, 
\begin{equation}\label{eq:mu-nu-relation} 
\nu_0 \circ \pi^{-1} = \mu_0 = \mu_{\varphi}. 
\end{equation}

Let $\epsilon \in (0,1/2)$. 
We assume that 
\begin{equation}\label{eq:r-max-min-epsilon} 
\epsilon \le \min_{i \in \mathcal{I}_N} r_i \le \max_{i \in \mathcal{I}_N} r_i \le 1-\epsilon. 
\end{equation}  
Since $\{f_i\}_{i \in \mathcal{I}_N}$ is a compatible system, $\displaystyle \sum_{i \in \mathcal{I}_N} r_i \ge 1$ and $r_i > 0$ for each $i \in \mathcal{I}_N$. 
Let $s \ge 1$ be the constant such that $\displaystyle \sum_{i \in \mathcal{I}_N} r_i^s = 1$. 

As the following assertion shows, 
the constant $s$ is close to $1$ if $\delta > 0$ is sufficiently small. 

\begin{Lem}\label{lem:delta-specify}
Let $\epsilon_2 > 0$. 
Then, for $\delta > 0$ such that $\displaystyle \sum_{i \in \mathcal{I}_N} r_i \le 1+\delta$ and 
\begin{equation}\label{eq:final-delta}  
\frac{\log N}{\log (N /(1+\delta))} < 1 + \frac{\epsilon_2}{\log (1/\epsilon)},  
\end{equation}
it holds that 
\[ s \le \frac{\log N}{\log (N/(1+\delta))} < 1+  \dfrac{\epsilon_2}{\log 1/\epsilon} \]
 for every choice of $(r_i)_{i \in \mathcal{I}_N} \in \mathcal{P}_N$ such that $\epsilon \le r_i \le 1-\epsilon$ for each $i \in \mathcal{I}_N$. 
\end{Lem}

\begin{proof}
Let $\displaystyle u_i \coloneqq \frac{r_i}{\sum_{i \in \mathcal{I}_N} r_i}$. 
Then, $\displaystyle \sum_{i \in \mathcal{I}_N} u_i = 1$.
We see that $\displaystyle \sum_{i \in \mathcal{I}_N} r_i^s = 1$ if and only if $\displaystyle \sum_{i \in \mathcal{I}_N} u_i^s = \left(\sum_i r_i \right)^{-s}$. 
Since $\displaystyle \sum_{i \in \mathcal{I}_N} r_i \le 1+\delta$, 
\[ \frac{\log\left(\sum_{i \in \mathcal{I}_N} u_i^s \right)}{s} \ge -\log(1+\delta). \] 

The sum $\displaystyle \sum_{i \in \mathcal{I}_N} u_i^s$ takes its minimum at $u_0 = \cdots = u_{N-1} = \dfrac{1}{N}$. 
We see this by the method of Lagrange multipliers. 
The minimum of the sum is $N^{1-s} \le 1$. 
Hence, 
$\displaystyle -\frac{s-1}{s} \log N \ge - \log \left(\sum_{i \in \mathcal{I}_N} r_i \right)$, which is equivalent to $s \le \dfrac{\log N}{\log (N /(1+\delta))}$. 
Therefore,  if \eqref{eq:final-delta} holds, 
then $\displaystyle s \le \frac{\log N}{\log (N/(1+\delta))} < 1+  \frac{\epsilon_2}{\log (1/\epsilon)}$.    
\end{proof}

Let 
\[ R_{n}\left({\bf q}; z\right) \coloneqq \frac{\nu_0 \left( I(\xi_1 (z), \ldots, \xi_n (z)) \right)}{q_{\xi_1(z)} \cdots q_{\xi_n (z)}}, \ \ {\bf q} = (q_i)_i \in \mathcal{P}_N, z \in (\mathcal{I}_N)^{\mathbb{N}}.  \]

By a measure-theoretic argument, we can show the following assertion: 
\begin{Lem}\label{lem:dimH}
Let $\widetilde{q}_i \coloneqq r_i^s > 0$ and ${\bf \widetilde{q}} \coloneqq (\widetilde{q}_i)_i \in \mathcal{P}_N$. 
If there exists $\epsilon_2 > 0$ such that 
\begin{equation}\label{eq:non-uniform-R}  
\limsup_{n \to \infty} \frac{-\log R_{n}({\bf \widetilde{q}};z)}{n}   \le -\epsilon_2, \ \textup{ $\nu_0$-a.s.$z  \in (\mathcal{I}_N)^{\mathbb{N}}$.}  
\end{equation}
then 
\[ \dim_H \mu_{\varphi}  \le s - \frac{\epsilon_2}{\log 1/\epsilon}. \]
\end{Lem}

We omit the proof of this assertion. 
See  the proof of \cite[Lemma 7]{Okamura2020}. 
The equation \eqref{eq:mu-nu-relation} plays an important role.

By Lemmas \ref{lem:delta-specify} and  \ref{lem:dimH}, 
we obtain the following assertion: 
\begin{Prop}\label{prop:uniform-delta}
Assume that 
there exists $\epsilon_2 > 0$ such that for every ${\bf q} \in \mathcal{P}_N$, 
\begin{equation}\label{eq:uniform-R}  
\limsup_{n \to \infty} \frac{-\log R_{n}({\bf q};z)}{n}   \le -\epsilon_2, \ \textup{ $\nu_0$-a.s.$z  \in (\mathcal{I}_N)^{\mathbb{N}}$.}  
\end{equation}
Then, 
for $\delta = \delta(N, \epsilon, \epsilon_2) > 0$ satisfying \eqref{eq:final-delta}, 
it holds that 
$\dim_H \mu_{\varphi} < 1$. 
\end{Prop}

Now it suffices to establish \eqref{eq:uniform-R}. 
We give some definitions for the relative entropy. 
Let 
\[ s_N \left( {\bf p}  | {\bf q} \right) \coloneqq \sum_{i \in \mathcal{I}_N} p_i \log \frac{q_i}{p_i}, \  \ {\bf p} = (p_i)_{i \in \mathcal{I}_N}, {\bf q} = (q_i)_{i \in \mathcal{I}_N} \in \overline{\mathcal{P}_N}. \]
If $p_i = 0$, then we let $p_i \log \dfrac{q_i}{p_i} \coloneqq 0$. 
This is also valid for $q_i = 0$. 
For each ${\bf q} \in \mathcal{P}_N$, the function $s_N ( \cdot | {\bf q} )$ is continuous on $\overline{\mathcal{P}_N}$.  
If ${\bf q} \in \overline{\mathcal{P}_N} \setminus \mathcal{P}_N$, $s_N$ can take $-\infty$. 
By Jensen's inequality, $s_N \left({\bf p} | {\bf q} \right) \le 0$. 
It holds that $s_N \left({\bf p} | {\bf q}\right) = 0$ if and only if ${\bf p} = {\bf q}$. 

For $y \in Y, z \in (\mathcal{I}_N)^{\mathbb{N}}$, let 
\[ p_j (y;z) \coloneqq \left( G_i \circ H_{\xi_{j}(z)} \circ \cdots \circ H_{\xi_1 (z)}(y) \right)_{i \in \mathcal{I}_N} \in \overline{\mathcal{P}_N}, \  j \ge 1. \]
and 
\[ p_0 (y;z) \coloneqq {\bf G}(y)  \in \overline{\mathcal{P}_N}. \]  

By a martingale argument, we can show the following assertion: 
\begin{Prop}\label{prop:reduction-rel-ent}
It holds that for every ${\bf q} \in \mathcal{P}_N$, 
\[ \limsup_{n \to \infty} \frac{-\log R_{n}({\bf q};z)}{n} = \limsup_{n \to \infty} \frac{1}{n} \sum_{j=0}^{n-1} s_N (p_j (0;z) | {\bf q}), \ \textup{ $\nu_0$-a.s.$z  \in (\mathcal{I}_N)^{\mathbb{N}}$.}  \] 
\end{Prop}

We omit the proof of this assertion. 
See the proof of \cite[Proposition 6]{Okamura2020}. 
It is important to assume that  ${\bf q} \in \mathcal{P}_N$, {\it not} ${\bf q} \in \overline{\mathcal{P}_N}$.

For $\epsilon \in (0,1)$, let 
\[ \overline{\mathcal{P}_N (\epsilon)} \coloneqq \left\{(p_i)_{i \in \mathcal{I}_N} \colon p_i \ge \epsilon, \sum_{i \in \mathcal{I}_N} p_i = 1 \right\}.\]  
Then, $\overline{\mathcal{P}_N (\epsilon)} \subset \mathcal{P}_N \subset \overline{\mathcal{P}_N}$.

Let $\widetilde{q}_i \coloneqq r_i^s > 0$ and ${\bf \widetilde{q}} \coloneqq (\widetilde{q}_i)_i \in \mathcal{P}_N$. 
Recall Lemma \ref{lem:delta-specify}. 
If 
\[ \frac{\log N}{\log (N/(1+\delta))} \le 1 + \frac{\log 2}{\log (1/\epsilon)}, \]
then, by \eqref{eq:r-max-min-epsilon},  
$\displaystyle r_i^s \ge \epsilon^s \ge \epsilon^{1 + \frac{\log 2}{\log (1/\epsilon)}} = \dfrac{\epsilon}{2}$ and hence, 
${\bf \widetilde{q}} \in \overline{\mathcal{P}_N (\epsilon/2)}$. 
In other words, ${\bf \widetilde{q}}$ has positive distance from the boundary $\overline{\mathcal{P}_N} \setminus \mathcal{P}_{N}$.

Assume the following assertion: 
\begin{Prop}\label{prop:uniform-Rn} 
If $\left|\{\mathfrak{p}_i\}_i \right| \ge 2$, then there exists $\epsilon_2 > 0$ such that for every ${\bf q} \in \overline{\mathcal{P}_N (\epsilon/2)}$, 
\begin{equation*}
\limsup_{n \to \infty} \frac{1}{n} \sum_{j=0}^{n-1} s_N \left(p_j (0;z) | {\bf q} \right) \le -\epsilon_2, \ \textup{ $\nu_0$-a.s.$z  \in (\mathcal{I}_N)^{\mathbb{N}}$.}  
\end{equation*}
\end{Prop}

Theorem \ref{thm:main-singular-1} follows from Propositions \ref{prop:uniform-delta}, \ref{prop:reduction-rel-ent} and \ref{prop:uniform-Rn}. 
We now show Proposition \ref{prop:uniform-Rn}. 
We give a strategy of the proof. 
By the definition, $s_N \left(p_j (0;z) | {\bf q} \right) \le 0$. 
Hence, it suffices to show the number of $j$ satisfying that $s_N \left(p_j (0;z) | {\bf q} \right)$ is strictly negative has positive density in $\mathbb{N}$. 

We first give some further definitions for quantities of the relative entropy. 
We remark that the diameter of $\overline{\mathcal{P}}_N$ with respect to $\| \cdot \|_1$ is equal to $2$. 
By the triangle inequality, 
it holds that for ${\bf p} = (p_i)_{i \in \mathcal{I}_N} , {\bf q} = (q_i)_{i \in \mathcal{I}_N}  \in \overline{\mathcal{P}}_N$, 
\[ \left\| {\bf p} - {\bf q} \right\|_1 = \sum_{i \in \mathcal{I}_N} |p_i - q_i| \le \sum_{i \in \mathcal{I}_N}  (p_i + q_i) = 2.  \] 
If ${\bf p} = (1,0,\dots, 0)$ and ${\bf q} = (0, \dots, 0, 1)$, then $\left\| {\bf p} - {\bf q} \right\|_1 = 2$.

\begin{Lem}\label{lem:V}
For ${\bf q} \in \mathcal{P}_N$, let 
\[ V\left({\bf q}; \epsilon^{\prime}\right) \coloneqq \max\left\{ s_N \left( {\bf p} | {\bf q} \right) \middle| {\bf p} \in \overline{\mathcal{P}_N}, \| {\bf p} - {\bf q} \|_1 \ge \epsilon^{\prime} \right\}, \ \epsilon^{\prime} \in (0,1].  \]
Then, \\
(i) $V\left({\bf q}; \epsilon^{\prime}\right) < 0$. \\
(ii) $V\left({\bf q}; \epsilon^{\prime}\right)$ is decreasing with respect to $\epsilon^{\prime}$.\\ 
(iii) $V\left({\bf q}; \epsilon^{\prime}\right)$ is right-continuous with respect to $\epsilon^{\prime}$. \\
(iv) $V\left({\bf q}; \epsilon^{\prime}\right)$ is continuous with respect to ${\bf q}$. 
\end{Lem}

\begin{proof}
Assertions (i) and (ii) are easy to see. 
(iii) Assume that for some $\epsilon_0 \in (0, 1/3)$, $V\left({\bf q}; \epsilon^{\prime}\right)$ is not right-continuous. 
Then, $\displaystyle V\left({\bf q}; \epsilon_0\right) > \lim_{\epsilon^{\prime} \to \epsilon_0 + 0} V\left({\bf q}; \epsilon^{\prime}\right)$. 
Assume that $V\left({\bf q}; \epsilon_0\right) = s_N \left( {\bf p} | {\bf q}  \right)$. 
If $\left\|{\bf p} - {\bf q} \right\|_1 > \epsilon_0$, then this cannot occur. 
Assume $\left\|{\bf p} - {\bf q} \right\|_1 = \epsilon_0$. 
Then, there exists a sequence $\{{\bf p}^{(n)}\}_n$ in $\overline{\mathcal{P}_N}$ such that $\left\|{\bf p}^{(n)} - {\bf q} \right\|_1 \ge \epsilon_0 + N^{-n}$ for each $n$ and 
$\displaystyle \lim_{n \to \infty} \left\|{\bf p}^{(n)} - {\bf q} \right\|_1 = \epsilon_0$. 
Hence, 
$$\liminf_{n \to \infty} V\left({\bf q}; \epsilon_0 + N^{-n}\right)  \ge \lim_{n \to \infty} s_N \left({\bf p}^{(n)} | {\bf q} \right) = s_N \left( {\bf p} | {\bf q}  \right) = V\left({\bf q}; \epsilon_0\right). $$
This is a contradiction. 

(iv) follows from (iii) and 
\[ V\left({\bf q}; \epsilon^{\prime} + \left\|{\bf p} - {\bf q}\right\|_1 \right) \le V\left({\bf p}; \epsilon^{\prime}\right). \qedhere\] 
\end{proof}

Setting 
\[ V_{\epsilon,1}(\epsilon^{\prime}) \coloneqq \max_{{\bf q} \in \overline{\mathcal{P}_N (\epsilon/2)}} V\left({\bf q}; \epsilon^{\prime}\right), \ \ \epsilon^{\prime} \in (0,1], \] 
we obtain by Lemma \ref{lem:V} and the compactness of $\overline{\mathcal{P}_N (\epsilon/2)}$ that 
\[ V_{\epsilon,1}(\epsilon^{\prime}) < 0.\]

\begin{Lem}\label{lem:basic-W} 
Let 
\[ W_i \left({\bf q}, y \right) \coloneqq \max\left\{\max_{k \in \mathcal{I}_N} |G_k (y) - q_k|, \max_{k \in \mathcal{I}_N} |G_k (H_i (y)) - q_k| \right\}, \ {\bf q} \in \overline{\mathcal{P}_N}, y \in Y, i \in \mathcal{I}_N. \]
Then, \\
(i) $W_i ({\bf q}, y)$ is continuous on $\overline{\mathcal{P}_N} \times Y$. \\
(ii) $W_i ({\bf q}, y) = 0$ if and only if ${\bf q} = \mathfrak{p}_i$ and $y = \textup{Fix}(H_i)$. 
\end{Lem}

This assertion is easily shown and we omit the proof.

If $\max_{k \in \mathcal{I}_N} |G_k (y) - q_k| > 0$, then 
$s_N \left( {\bf G}(y) | {\bf q} \right) < 0$. 
However, in our setting, we cannot exclude the possibility that $|G_k (y) - q_k| = 0$ for every $k \in \mathcal{I}_N$. 
Assume that $\left|\{\mathfrak{p}_i\}_i \right| \ge 2$. 
If $\displaystyle \max_{k \in \mathcal{I}_N} \left|G_k (y) - q_k \right|$ is close to $0$, 
then $\displaystyle \max_{k \in \mathcal{I}_N} \left|G_k (H_i (y)) - q_k \right|$ is strictly positive.   
Hence, $W_i \left({\bf q}, y \right) > 0$. 

Since the diameter of $\overline{\mathcal{P}}_N$ is equal to $2$, we obtain $0 \le W_i \left({\bf q}, y \right) \le 2$.

The following  corresponds to \cite[Lemma 8]{Okamura2020}. 
Informally speaking, if $W_i ({\bf q}, y)$ is positive uniformly in $y$, 
then a sum of two successive terms of $s_N$ is negative. 

\begin{Lem}\label{lem:classify-multiple}
Let ${\bf q} \in \overline{\mathcal{P}_N (\epsilon/2)}$. 
Let $i_1 \in \mathcal{I}_N$. 
Assume that for some $\epsilon_1 \in (0,1]$, 
\begin{equation}\label{eq:W-lower-uniform}  
\inf_{y \in Y} W_{i_1} \left({\bf q}, y\right) \ge \epsilon_1. 
\end{equation} 
Assume that $z \in \left(\mathcal{I}_{N}\right)^{\mathbb{N}}$ and $\xi_{j+1}(z) = i_1$. 
Then, 
\[ s_N \left(p_{j}(0;z) | {\bf q} \right) +  s_N \left(p_{j+1}(0;z) | {\bf q} \right) \le V\left({\bf q}; \epsilon_1\right) \le V_{\epsilon,1}(\epsilon_1) < 0. \] 
\end{Lem}

\begin{proof}
The proof is essentially the same as the proof of \cite[Lemma 8]{Okamura2020}\footnote{The proof contains a typo. In Case 2 in the proof, $G_{i+l}(h_{i+l}(y;x))$ should be replaced with $G_{i_{l}}(h_{i+l}(y;x))$.}. 
Since $\xi_{j+1}(z) = i_1$, 
\[ p_{j}(0;z) = {\bf G}\left(H_{\xi_{j}(z)} \circ \cdots \circ H_{\xi_1(z)}(0)\right), \] 
and 
\[ p_{j+1}(0;z) = {\bf G}\left(H_{i_1} \circ H_{\xi_{j}(z)} \circ \cdots \circ H_{\xi_1(z)}(0)\right).   \] 
Let ${\bf q} = (q_k)_k$. 
Then, by \eqref{eq:W-lower-uniform}, 
either 
\[ \left|G_{k_1} (H_{\xi_{j}(z)} \circ \cdots \circ H_{\xi_1(z)}(0)) - q_{k_1} \right| \ge \epsilon_1, \textup{ for some } k_1 \in \mathcal{I}_N \] 
or 
\[ \left|G_{k_2} ( H_{i_1} \circ H_{\xi_{j}(z)} \circ \cdots \circ H_{\xi_1(z)}(0)) - q_{k_2} \right| \ge \epsilon_1, \textup{ for some } k_2 \in \mathcal{I}_N, \] 
holds. 
We remark that at least one of them holds. 
If the former holds, then 
$s_N \left(p_{j}(0;z) | {\bf q} \right) \le V\left({\bf q}; \epsilon_1\right)$.  
If the latter holds, then 
$s_N \left(p_{j+1}(0;z) | {\bf q} \right) \le V\left({\bf q}; \epsilon_1\right)$.  
Now it suffices to recall that $s_N$ takes non-positive values. 
\end{proof}

Let $\ell \coloneqq \left|\{\mathfrak{p}_i\}_i \right| \ge 2$. 
We assume that $\{\mathfrak{p}_i\}_i = \{\mathfrak{p}_{j_{k}} \colon 1 \le k \le \ell\}$ where $0 \le j_1 < \cdots < j_{\ell} \le N-1$. 
Let 
\begin{align*} 
\delta_0 &\coloneqq \min\left\{ \left\| \mathfrak{p}_{j_{k_1}} - \mathfrak{p}_{j_{k_2}} \right\|_1 \colon k_1 \ne k_2 \right\} > 0, \ \  \textup{ and } \\
B_j &\coloneqq \left\{{\bf p} \in \overline{\mathcal{P}_N} \colon \left\| {\bf p} - \mathfrak{p}_{j} \right\|_1 \le \delta_0/4 \right\}. 
\end{align*}  
The equality $B_j = B_k$ can occur even if $j \ne k$. 
Define a coloring map $\mathscr{A} \colon \overline{\mathcal{P}_N} \to \mathcal{I}_N$ by 
\[ \mathscr{A}\left( {\bf p} \right) \coloneqq \begin{cases} j_2 &  {\bf p} \in B_{j_1} \\  \vdots &  \\ j_{\ell}  & {\bf p} \in B_{j_{\ell -1}} \\ j_1  & {\bf p} \in B_{j_{\ell}} \\ 0 & \textup{ otherwise} \end{cases}.  \]

Under these definitions, we see the following assertion: 
\begin{Lem}\label{lem:non-zero-multiple}
\[ \widetilde{\epsilon_1} \coloneqq \inf\left\{ W_{\mathscr{A}\left( {\bf q} \right)}({\bf q}, y) \middle| {\bf q} \in \overline{\mathcal{P}_N}, y \in Y \right\} > 0.  \]
\end{Lem}

\begin{proof}
Fix ${\bf q} \in \overline{\mathcal{P}_N}$. 
By Lemma \ref{lem:basic-W}, 
\[ W_{ \mathscr{A}\left( {\bf q} \right)} ({\bf q}, y) \ge \frac{1}{2} \max_{j \in \mathcal{I}_N} \left|G_j (H_{\mathscr{A}\left( {\bf q} \right)} (y)) - G_j (y)\right|.  \]

Let 
\[ \widetilde{B}_j \coloneqq \left\{{\bf p} \in \overline{\mathcal{P}_N} \colon \left\| {\bf p} - \mathfrak{p}_{j} \right\|_1 \le \delta_{0}/8 \right\}, \ j \in \mathcal{I}_N. \]

If ${\bf G}(y) \notin \widetilde{B}_{\mathscr{A}({\bf q})}$, then $y \ne \mathfrak{p}_{\mathscr{A}({\bf q})} = {\bf G} \left(\textup{Fix}(H_{\mathscr{A}\left( {\bf q} \right)})\right)$
and hence, 
there exists an open interval $J_{\mathscr{A}({\bf q})}$ in $Y$ containing $\textup{Fix}(H_{\mathscr{A}\left( {\bf q} \right)})$ but not $y$. 
Since $G_j$ and $H_{\mathscr{A}\left( {\bf q} \right)}$ are continuous, 
$\displaystyle \max_{j \in \mathcal{I}_N} \left|G_j (H_{\mathscr{A}\left( {\bf q} \right)} (y)) - G_j (y)\right|$ is continuous with respect to $y$ on $Y$. 
Hence, 
\[ \inf_{{\bf G}(y) \notin \widetilde{B}_{\mathscr{A}({\bf q})}} \max_{j \in \mathcal{I}_N} \left|G_j (H_{\mathscr{A}\left( {\bf q} \right)} (y)) - G_j (y)\right| \] 
\[ \ge \inf_{y \notin J_{\mathscr{A}({\bf q})}} \max_{j \in \mathcal{I}_N} \left|G_j (H_{\mathscr{A}\left( {\bf q} \right)} (y)) - G_j (y)\right| =: \epsilon_{\mathscr{A}\left( {\bf q} \right)} > 0.\]
 
Assume that  ${\bf G}(y) \in \widetilde{B}_{\mathscr{A}({\bf q})}$. 
First, we consider the case where $\displaystyle {\bf q} \in \bigcup_{j \in \mathcal{I}_N} B_j$. 
Then, there exists a unique $i \in \{1, \dots, \ell\}$ such that ${\bf q} \in B_{j_i}$ and $\mathscr{A}({\bf q}) \ne j_i$. 
By the triangle inequality, 
\[ \left\| G(y) - {\bf q} \right\|_1 \ge \left\| \mathfrak{p}_{\mathscr{A}({\bf q})} - \mathfrak{p}_{j_i} \right\|_1 - \left\| G(y) -  \mathfrak{p}_{\mathscr{A}({\bf q})}  \right\|_1 -  \left\|  {\bf q} - \mathfrak{p}_{j_i} \right\|_1 \ge \frac{\delta_0}{2}.  \]
Second, we consider the case where $\displaystyle {\bf q} \notin \bigcup_{j \in \mathcal{I}_N} B_j$. 
Then, ${\bf q} \notin B_{\mathscr{A}({\bf q})}$. 
By the triangle inequality, 
\[ \left\| G(y) - {\bf q} \right\|_1 \ge \left\|\mathfrak{p}_{\mathscr{A}({\bf q})} - {\bf q} \right\|_1 - \left\| \mathfrak{p}_{\mathscr{A}({\bf q})} - G(y) \right\|_1 \ge \frac{\delta_0}{8}. \]

Thus we obtain that 
\[ W_{ \mathscr{A}\left( {\bf q} \right)} \left({\bf q}, y\right) \ge \frac{1}{N} \sum_{j \in \mathcal{I}_N} |G_j (y) - q_j| \ge \frac{\delta_0}{8N}. \]
Hence, 
\[  \inf_{y \in Y} W_{ \mathscr{A}\left( {\bf q} \right)} ({\bf q}, y) \ge \min\left\{\frac{\delta_0}{8N}, \epsilon_{j_k} \middle| 1 \le k \le \ell \right\} =: \widetilde{\epsilon_1}. \qedhere \] 
\end{proof}

We discuss the frequency of occurrence of some elements in $\mathcal{I}_N$ in $(\xi_n)_n$. 

\begin{Lem}\label{lem:azuma-apply}
Let $I \subset \mathcal{I}_N$ be a non-empty set. 
Let 
$\displaystyle c_{I} \coloneqq \inf_{y \in Y} \sum_{i \in I} G_i (y)$. 
Then, 
\[ \liminf_{n \to \infty} \frac{|\{1 \le j \le n \colon \xi_j (z) \in I \}|}{n} \ge c_I, \ \textup{ $\nu_0$-a.s. $z  \in (\mathcal{I}_N)^{\mathbb{N}}$.}  \] 
\end{Lem}

This corresponds to \cite[Lemma 9]{Okamura2020}, but the formulation differs.   

\begin{proof}
If $c_I = 0$, then the assertion is obvious. 
Assume $c_I > 0$. 
We regard $\{\xi_j\}_{j \ge 1}$ as random variables with respect to the probability measure $\nu_0$. 
We denote the expectation with respect to $\nu_0$ by $E^{\nu_0}$. 
Let $\mathcal{F}_m \coloneqq \sigma(\xi_1, \ldots, \xi_m)$ for $m \ge 1$ and $\mathcal{F}_0$ be the trivial $\sigma$-algebra on $(\mathcal{I}_N)^{\mathbb{N}}$. 
We denote the indicator function of $A  \subset (\mathcal{I}_N)^{\mathbb{N}}$ by ${\bf 1}_{A}$. 

Let $M_0 \coloneqq 0$ and 
$$M_j - M_{j-1} \coloneqq {\bf 1}_{\{\xi_{j} \in I\}} - E^{\nu_0}\left[  {\bf 1}_{\{\xi_{j} \in I\}} | \mathcal{F}_{j-1} \right], \ j \ge 1. $$
Then, $\left\{M_n, \mathcal{F}_n \right\}_n$ is a martingale. 

By \eqref{eq:def-nu}, for every $j \ge 1$, 
$E^{\nu_0}\left[  {\bf 1}_{\{\xi_j \in I\}} | \mathcal{F}_{j-1} \right] \ge c_I$, $\nu_0$-a.s. 
Then, we see that 
\[ \frac{1}{n} \sum_{i=1}^{n} {\bf 1}_{I}(\xi_i) = \frac{M_n}{n} + \frac{1}{n} \sum_{i=1}^{n} E^{\nu_0}\left[  {\bf 1}_{\{\xi_j \in I\}} | \mathcal{F}_{j-1} \right] \ge \frac{M_n}{n} + c_I, \textup{ $\nu_0$-a.s.}   \]
Therefore, it holds that for every $\eta \in (0,1)$, 
\begin{align*} 
\nu_0 \left( \frac{|\{1 \le j \le n \colon \xi_j  \in I \}|}{n} \le (1-\eta) c_I \right) &\le \nu_0 \left(  \frac{M_n}{n} + c_I \le (1-\eta) c_I \right)  \\
&= \nu_0 \left( M_n \le - n \eta c_I \right). 
\end{align*} 
We remark that $|M_n - M_{n-1}| \le 1$ for each $n \ge 1$. 
By Azuma's inequality (\cite{Azuma1967}, \cite[E14.2]{Williams1991}), 
\[ \nu_0 \left( M_n \le -n \eta c_I \right) \le \exp\left( - \frac{n \eta^2 c_I^2}{2} \right). \]
Using $c_I > 0$ and Borel-Cantelli's lemma, 
\[ \liminf_{n \to \infty} \frac{|\{1 \le j \le n \colon \xi_j (z) \in I \}|}{n} \ge (1-\eta) c_I, \ \textup{ $\nu_0$-a.s.$z  \in (\mathcal{I}_N)^{\mathbb{N}}$.}  \]
The assertion follows for $\eta \to +0$. 
\end{proof}

We proceed to the proof of Proposition \ref{prop:uniform-Rn}. 
The proof is divided into two cases. 
The following is easily shown. 
\begin{Lem}\label{lem:alpha-beta}
We have the following assertions: \\ 
(i) $a_0 \le 1$ and $b_{N-1} + c_{N-1} \ge 0$.\\
(ii) If $a_0 = 1$ and $b_{N-1} + c_{N-1} = 0$, then $\alpha = -1, \beta = +\infty$.\\
(iii) If $a_0 = 1$ and $b_{N-1} + c_{N-1} > 0$, then $\alpha > -1, \beta =  +\infty$.\\
(iv) If $a_0 < 1$ and $b_{N-1} + c_{N-1} = 0$, then $\alpha = -1, \beta <  +\infty$.\\
(v) If $a_0 < 1$ and $b_{N-1} + c_{N-1} > 0$, then $\alpha > -1, \beta <  +\infty$.\\
(vi) If $a_0 = 1$, then $\textup{Fix}(H_0) = \{+\infty\}$ and $\mathfrak{p}_0 = (1,0,\ldots, 0) \in \overline{\mathcal{P}_N} \setminus \mathcal{P}_N$.\\
(vii) If $b_{N-1} + c_{N-1} = 0$, then $\textup{Fix}(H_{N-1}) = \{-1\}$ and $\mathfrak{p}_{N-1} = (0,\ldots, 0,1) \in \overline{\mathcal{P}_N} \setminus \mathcal{P}_N$. 
\end{Lem}

$\alpha = -1$ or $\beta = +\infty$ occurs according to the values of $a_0$ and $b_{N-1} + c_{N-1}$. 

{\it Case 1}. We first deal with the case where  $a_0  < 1$ and $b_{N-1} + c_{N-1} > 0$.

\begin{proof}[Proof of Proposition \ref{prop:uniform-Rn}, Case 1]
By the assumption that $a_0  < 1$ and $b_{N-1} + c_{N-1} > 0$ and Lemma \ref{lem:alpha-beta} (v), $-1 < \alpha \le \beta < +\infty$. 
Hence, 
\[ 0 <  \inf_{y \in Y, i \in \mathcal{I}_N} G_i (y)  \le \sup_{y \in Y, i \in \mathcal{I}_N} G_i (y) < 1. \]
Fix ${\bf q}  \in \overline{\mathcal{P}_N (\epsilon/2)}$. 
Let $\displaystyle \widetilde{c} \coloneqq  \inf_{y \in Y, i \in \mathcal{I}_N} G_i (y)$. 
Then, 
\[ c_{\{\mathscr{A}({\bf q})\}} = \inf_{y \in Y}  G_{\mathscr{A}({\bf q})} (y) \ge \widetilde{c} > 0, \] 
and by Lemma \ref{lem:azuma-apply} for $I = \{\mathscr{A}({\bf q})\}$,   
\[ \liminf_{n \to \infty} \frac{1}{n} \sum_{k=1}^{n} {\bf 1}_{\{\xi_k = \mathscr{A}({\bf q})\}}(z) \ge \widetilde{c}, \  \textup{ $\nu_0$-a.s. $z  \in (\mathcal{I}_N)^{\mathbb{N}}$.}  \] 

The number $i_1 \in \mathcal{I}_N$ in  Lemma \ref{lem:classify-multiple} will be chosen according to the position of ${\bf q}$. 
By applying Lemma \ref{lem:classify-multiple} to $i_1 = \mathscr{A}({\bf q})$ and $\epsilon_1 = \widetilde{\epsilon_1} $ where $\widetilde{\epsilon_1} $ is the constant in Lemma \ref{lem:non-zero-multiple}, 
we see that if $\xi_{i+1}(z) =  \mathscr{A}({\bf q})$, then 
\[ s_N \left(p_{i}(0;z) | {\bf q} \right) +  s_N \left(p_{i+1}(0;z) | {\bf q} \right) \le V_{\epsilon,1}(\widetilde{\epsilon_1}) < 0. \] 
Now the assertion follows if we let $\epsilon_2 \coloneqq -\dfrac{\widetilde{c} \ V_{\epsilon,1}(\widetilde{\epsilon_1})}{2} > 0$. 
\end{proof}

{\it Case 2}. \ We now deal with the case where $a_0  = 1$ or $b_{N-1} + c_{N-1} = 0$. 

If $a_0 = 1$ and $b_{N-1} + c_{N-1} = 0$, then $c_{\{0\}} = c_{\{N-1\}} = 0$, so we consider $c_{\{0,N-1\}}$. 
In the following, we use neither Lemma \ref{lem:non-zero-multiple} nor the assumption that $\left|\left\{\mathfrak{p}_i \right\}_i \right| \ge 2$.

\begin{proof}[Proof of Proposition \ref{prop:uniform-Rn}, Case 2]

We consider the following three cases.

{\it Case 2.1.}  We consider the case where $a_0 = 1$ and $b_{N-1} + c_{N-1} > 0$. 

By Lemma \ref{lem:alpha-beta}, 
we see that 
$\alpha > -1$ and $\beta = +\infty$. 
It holds that $H_0 (y) = y + c_0$, and hence $\textup{Fix}(H_0) = \{+\infty\}$. 
Hence, $\mathfrak{p}_0 = {\bf G}(\textup{Fix}(H_{0})) = (1,0, \ldots, 0) \notin \mathcal{P}_N$. 
On the other hand, $\mathfrak{p}_i \in \mathcal{P}_N$ for $i \ge 1$. 

Since $G_0 (+\infty) = G_0 (H_0 (+\infty)) =  1$ and $q_0 \le 1-\epsilon/2$, 
$$\inf_{y \in Y} W_0 ({\bf q}, y) \ge \widetilde{c_0} (\epsilon/2),$$ 
where the positive lower bound $\widetilde{c_0} (\epsilon/2)$ is uniform with respect to ${\bf q} \in \overline{\mathcal{P}_N (\epsilon/2)}$.  

Since $\alpha > -1$, 
it holds that $\displaystyle \inf_{y \in Y} G_0 (y) \ge c$ for some  $c > 0$. 
By using Lemma \ref{lem:classify-multiple} and Lemma \ref{lem:azuma-apply} for $I = \{0\}$,  
we have the assertion. 

{\it Case 2.2.} We consider the case where $a_0 < 1$ and $b_{N-1} + c_{N-1} = 0$. 

By Lemma \ref{lem:alpha-beta}, we see that $\alpha = -1$ and $\beta < +\infty$. 
It holds that $H_{N-1} (y) = \dfrac{(1-2b_{N-1})y+b_{N-1}}{1-b_{N-1}}$, and hence $\textup{Fix}(H_{N-1}) = \{-1\}$. 
Hence, $\mathfrak{p}_{N-1} = {\bf G}(\textup{Fix}(H_{N-1})) = (0, \ldots, 0,1) \notin \mathcal{P}_N$. 
On the other hand, $\mathfrak{p}_i \in \mathcal{P}_N$ for $i \le N-2$. 

Since $G_{N-1} (-1) = G_{N-1} (H_{N-1} (-1)) =  1$ and $q_{N-1} < 1-\epsilon/2$, 
$$\inf_{y \in Y} W_{N-1} ({\bf q}, y) \ge \widetilde{c_{N-1}} (\epsilon/2), $$
where the positive lower bound $\widetilde{c_{N-1}} (\epsilon)$ is uniform with respect to ${\bf q} \in \overline{\mathcal{P}_N (\epsilon/2)}$.  
Since $\beta < +\infty$, 
it holds that $\displaystyle \inf_{y \in Y} G_{N-1} (y) \ge c$ for some $c > 0$. 
The rest of the argument is the same as in the above case.

{\it Case 2.3.} We consider the case where $a_0 = 1$ and $b_{N-1} + c_{N-1} = 0$. 

Since $\textup{Fix}(H_0) = \{+\infty\}$, $\textup{Fix}(H_{N-1}) = \{-1\}$, $q_0 \le 1-\epsilon/2$ and $q_{N-1} \le 1-\epsilon/2$, 
it holds that 
$$\inf_{y \in Y} W_0 ({\bf q}, y) > \widetilde{c_0} (\epsilon/2) > 0$$ and 
$$\inf_{y \in Y} W_{N-1} ({\bf q}, y) > \widetilde{c_{N-1}} (\epsilon/2) > 0.$$ 
Since 
\[ G_0 (y) + G_{N-1}(y) = \frac{b_1 b_{N-1} y^2 + 2b_1 y  + 1+ b_1 - b_{N-1}}{(b_1 y + 1)(b_{N-1}y+1)}, \] 
it holds that 
$$\inf_{y \in Y} G_0 (y) + G_{N-1}(y) > 0.$$
By using Lemma \ref{lem:classify-multiple} and  Lemma \ref{lem:azuma-apply} for $I = \{0,N-1\}$,   
we have the assertion. 
\end{proof}

Now the proof of Theorem \ref{thm:main-singular-1} is completed.

\begin{Rem}
(i) In the proof of  \cite[Theorem 43]{Okamura2020}, the details of Case 2 were not given. \\
(ii) Assume that $\{g_i\}_{i \in \mathcal{I}_N}$ are affine maps such that  $g_i (x) = b_i x + (b_0 + \cdots + b_{i-1})$. 
Since $Y = \{0\}$, there is no need to introduce $H_i$, and $W_i ({\bf q}, y) = \max_k |G_k (y) - q_k| = \max_k |b_k - q_k|$.
However, it holds that $\mathfrak{p}_0 = \cdots = \mathfrak{p}_{N-1} = (b_i)_i$, and hence this is not in the scope of Theorem \ref{thm:main-singular-1}. 
This case is considered in Theorem \ref{thm:main-singular-2}.  
\end{Rem}

We now proceed to the proof of Theorem \ref{thm:main-singular-2}. 

By Lemma \ref{lem:delta-specify}, 
if we take $\delta > 0$ such that 
\begin{equation}\label{eq:delta-2} 
\frac{\log N}{\log (N /(1+\delta))}\le 1 + \frac{\log\left(1 + \epsilon/2\right)}{\log (1/\epsilon)}, 
\end{equation}   
then $\displaystyle s \le 1 + \dfrac{\log\left(1 + \epsilon/2\right)}{\log (1/\epsilon)}$. 

Assume that $\displaystyle s \le 1 + \dfrac{\log\left(1 + \epsilon/2\right)}{\log (1/\epsilon)}$. 
Since $\displaystyle \min_{i \in \mathcal{I}_N} r_i \ge \epsilon$, 
$\displaystyle \sum_{i \in \mathcal{I}_N} |r_i - r_i^s| \le \epsilon^{1-s} - 1 \le \epsilon/2$. 
Hence, $\left\|{\bf \widetilde{q}} - \mathfrak{p}_0 \right\|_1 \ge \epsilon/2$, 
where we let $\widetilde{q}_i \coloneqq r_i^s$ for each $i \in \mathcal{I}_N$ and ${\bf \widetilde{q}} \coloneqq \left(\widetilde{q}_i\right)_{i \in \mathcal{I}_N}$. 
Since $\epsilon \in (0,1)$, $\widetilde{q}_i \ge \epsilon^s \ge \dfrac{\epsilon}{1+\epsilon/2} \ge \epsilon/2$ for each $i \in \mathcal{I}_N$.  
Hence, ${\bf \widetilde{q}} \in \overline{\mathcal{P}_N (\epsilon/2)}$. 

By Lemma \ref{lem:basic-W} (i) and (ii), 
it holds that 
\[ \epsilon_3 \coloneqq \inf \left\{ W_i ({\bf q}, y) \middle| i \in \mathcal{I}_N, y \in Y, {\bf q} \in \overline{\mathcal{P}_N (\epsilon/2)}, \left\|{\bf q} - \mathfrak{p}_0 \right\|_1 \ge \epsilon/2 \right\} > 0. \]

By Lemma \ref{lem:classify-multiple}, 
we see that for every $j \ge 0$, every $z \in (\mathcal{I}_N)^{\mathbb{N}}$, and every ${\bf q} \in \overline{\mathcal{P}_N (\epsilon/2)}$ with $\left\|{\bf q} - \mathfrak{p}_0 \right\|_1 \ge \epsilon/2$, 
\[ s_N \left(p_{j}(0;z) | {\bf q} \right) +  s_N \left(p_{j+1}(0;z) | {\bf q} \right) \le V\left({\bf q}; \epsilon_3\right). \] 

By Lemma \ref{lem:V}, we see that 
\[ V_{\epsilon,2}(\epsilon_3) \coloneqq \max\left\{ V\left({\bf q}; \epsilon_3\right) \middle| {\bf q} \in \overline{\mathcal{P}_N (\epsilon/2)}, \left\|{\bf q} - \mathfrak{p}_0 \right\|_1 \ge \epsilon/2 \right\} < 0. \]

Then, 
\[ \limsup_{n \to \infty} \frac{1}{n} \sum_{j=0}^{n-1} s_N \left(p_j (0;z) | {\bf q} \right) \le \frac{ V_{\epsilon,2}(\epsilon_3)}{2}, \ \textup{ $\nu_0$-a.s.  $z  \in (\mathcal{I}_N)^{\mathbb{N}}$.}  \]

Assume that $\delta > 0$ satisfies not only \eqref{eq:delta-2} but also  \eqref{eq:final-delta} for $\epsilon_2 \coloneqq V_{\epsilon}(\epsilon_3)/2$. 
Then, by Propositions \ref{prop:uniform-delta} and \ref{prop:reduction-rel-ent}, $\dim_H \mu_{\varphi} < 1$.

\begin{Rem}
(i) If $\mathfrak{p}_{0} = \cdots = \mathfrak{p}_{N-1}$, then $\mathfrak{p}_{0} \in \mathcal{P}_N$.  
We show this. 
Since $G_0 (y) = \dfrac{b_1 (y+1)}{b_1 y + 1}$ and $0 < b_1 < 1$, it is injective on $Y$. 
Hence, $\textup{Fix}(H_0)  = \cdots = \textup{Fix}(H_{N-1})$. 
Hence, $\displaystyle \frac{c_0}{1-a_0} = \frac{a_1 - 1 + \sqrt{(1-a_1)^2 + 4b_1 c_1}}{2 b_1} < \infty$. 
By Definition \ref{def:lf-family} and the assumption that $d_0 = 1$, 
we see that 
$a_0 \in (0,1)$ and $\dfrac{a_0}{c_0 + 1} = b_1 \in (0,1)$. 
Hence, $\displaystyle \frac{c_0}{1-a_0} > -1$. 
Recall that $\displaystyle G_i (y) = \frac{(b_{i+1} - b_{i})(y+1)}{(b_i y + 1)(b_{i+1} y + 1)}, i \in \mathcal{I}_N$. 
Since $0 = b_0 < b_1 < \cdots < b_{N-1} < b_N = 1$, $G_i (y) > 0$ for every $y \in (-1,+\infty)$ and every $i \in \mathcal{I}_N$. 
Hence, $\displaystyle \mathfrak{p}_{0} = \left(G_i \left( \frac{c_0}{1-a_0} \right) \right)_{i \in \mathcal{I}_N} \in \mathcal{P}_N$.\\ 
(ii) In the proof of Theorem \ref{thm:main-singular-2}, 
we do not need to use Lemma \ref{lem:azuma-apply}, since we have the estimate that $\epsilon_3 > 0$.  
We remark that for every small $\epsilon > 0$ and every $i \in \mathcal{I}_N$, 
\[ \min \left\{ W_i ({\bf q}, y) \middle| y \in Y, {\bf q} \in \overline{\mathcal{P}_N (\epsilon/2)} \right\}  = 0,  \]

because, by (i) above, we can let ${\bf q} = {\bf G}\left( \textup{Fix}(H_i) \right) = \mathfrak{p}_i = \mathfrak{p}_{j}, j \in \mathcal{I}_N$ and $y = \textup{Fix}(H_i)$. \\ 
(iii) The case where both $\{f_i\}_{i \in \mathcal{I}_N}$ and $\{g_i\}_{i \in \mathcal{I}_N}$ are affine maps is easily handled by the dimension formula. 
However, if some of the functions $\{f_i, g_i\}_i$ are non-affine maps, 
then it is difficult to judge the singularity by the generalized dimension formula in \cite[Corollary 3.5]{Fan1999}. 
\end{Rem}

\section{Examples}\label{sec:ex} 

We give three examples assuming that $\{g_i\}_{i \in \mathcal{I}_N}$ is an LF system and $\{f_i\}_{i \in \mathcal{I}_N}$ are non-affine maps and the solution $\varphi$ is singular.   
The result depends on the quantitative estimates in Section \ref{sec:singu-non-linear}. 
We assume that $d_i = 1$ for every $i \in \mathcal{I}_N$.

We first deal with the case where $N=2$. 
Then, 
\[ G_0 (y) = \frac{b_1 (y+1)}{b_1 y + 1}, \ H_0 (y) = a_0 y + c_0, \ H_1 (y) = \frac{(1-b_1+c_1) y + c_1}{b_1 y + 1}. \]

\begin{Exa}
Let 
$$g_0 (x) \coloneqq \frac{x}{-x+6} = \frac{\frac{x}{6}}{-\frac{x}{6} + 1}, \textup{ and }  g_1 (x) \coloneqq \frac{3x+1}{5-x} = \frac{\frac{3}{5} x + \frac{1}{5}}{-\frac{x}{5} + 1}.$$ 
Then, $g_0$ is contractive, and $g_1$ is weakly contractive but not contractive.   
We see that 
\[ \alpha = \min\left\{0, -\frac{1}{5}, -1\right\} = -1, \ \beta = \left\{0, -\frac{1}{5}, -1\right\} = 0. \]
We see that 
\[ G_0 (y) = \frac{y+1}{y+5}, \textup{ and } G_1 (y) = 1 - G_0 (y) = \frac{4}{y+5}. \]
Hence, 
\[ \sup_{y \in [-1,0]} G_0 (y) = 1 - \inf_{y \in [-1,0]} G_1 (y) =  \frac{1}{5}. \] 
It also holds that 
\[ H_0 (y) = \frac{y-1}{6} \textup{ and } H_1 (y) = \frac{3y-1}{y+5}.\]
Hence, $\textup{Fix}(H_0) = \{-1/5\}$ and $\textup{Fix}(H_1) = \{-1\}$. 
Hence, 
\[ {\bf G}\left(\textup{Fix}(H_0)\right) = \left(\frac{1}{6}, \frac{5}{6}\right) \textup{ and } {\bf G}(\textup{Fix}(H_1)) = (0,1). \]

Let $\widetilde{\mathcal{P}_2} \coloneqq \left\{(p_0, p_1) \in \mathcal{P}_2  \colon  p_0 \ge 4/5 \right\}$. 
Then, 
\[ \left|G_0 (y) - q_0 \right| = \left|G_1 (y) - q_1 \right| \ge \frac{3}{5}, \ y \in [-1,0], \ {\bf q} = (q_0, q_1) \in \widetilde{\mathcal{P}_2}.   \]

Since $s_2 \left({\bf p} | {\bf q} \right) = s_2 \left((p_0, 1-p_0)|(q_0, 1-q_0) \right)$ is increasing with respect to $p_0$ on $[0, q_0]$, 
\begin{align}\label{eq:direct-upper-bound-ent}  
s_2 \left({\bf G}(y) | {\bf q}\right) \le s_2 \left(\left(\frac{1}{5}, \frac{4}{5} \right)\middle| (q_0, 1-q_0) \right) &\le  s_2 \left(\left(\frac{1}{5}, \frac{4}{5}\right) \middle| \left(\frac{4}{5}, \frac{1}{5}\right) \right) \notag\\
&\le -\frac{6}{5} \log 2 \le -0.83.   
\end{align} 

By Proposition \ref{prop:reduction-rel-ent}, 
it holds that  
\begin{equation}\label{eq:upper-ex-numerical}
\limsup_{n \to \infty} \frac{-\log R_{n}({\bf q};z)}{n} = \limsup_{n \to \infty} \frac{1}{n} \sum_{j=0}^{n-1} s_N \left(p_j (0;z) | {\bf q} \right) \le -0.83, \ \textup{  $\nu_0$-a.s. $z$, \ ${\bf q}  \in \widetilde{\mathcal{P}_2}$.} 
\end{equation}

Let $\epsilon \in (0,1/4)$ and let $\{f_0, f_1\}$ be a contractive system such that $\epsilon \le r_i \le 1-\epsilon$, where $r_i \coloneqq \|f_i\|_{\textup{Lip}}$, $i=0,1$.   
Let $s \ge 1$ be the constant such that $r_0^s + r_1^s = 1$. 
Assume that $(r_0^s, r_1^s) \in \widetilde{\mathcal{P}_2}$. 
Then, by \eqref{eq:upper-ex-numerical} and Lemma \ref{lem:dimH}, 
\[ \dim_{H} \mu_{\varphi} \le s - \frac{0.83}{\log (1/\epsilon)}. \]

By Lemma \ref{lem:delta-specify}, if $\{f_0, f_1\}$ is a contractive system such that $0.1 \le r_i \le 0.9$, $i=0,1$, and $r_0 + r_1 \le 1.2$, 
then $s \le \dfrac{\log 2}{\log 2 - \log 1.2}$. 

Since we see that 
$\dfrac{0.83}{\log 10} \ge 0.36$ and $\dfrac{\log 2}{\log 2 - \log 1.2} \le 1.357$, 
equation \eqref{eq:final-delta} holds for $N=2$, $\epsilon = 0.1$, $\epsilon_2 = 0.83$ and $\delta = 0.2$. 

Therefore, if $(r_0^s, r_1^s) \in \widetilde{\mathcal{P}_2}$, then 
\[ \dim_{H} \mu_{\varphi} \le s - \frac{0.83}{\log (1/\epsilon)} \le \dfrac{\log 2}{\log 2 - \log 1.2} - \frac{0.83}{\log 10}  < 1. \]

Now we can let $\{f_0, f_1\}$ be a compatible system such that $0.85 \le r_0 \le 0.9$, $0.1 \le r_1 \le 0.3$ and $r_0 + r_1 \ge 1$. 
Then, $r_0 + r_1 \le 1.2$.   
The maps $f_0$ and $f_1$ can be non-affine maps. 
Since $s \le \dfrac{\log 2}{\log 2 - \log 1.2} \le 1.357$, we have $r_0^s \ge (0.85)^{1.357} \ge 0.8$. 
Hence, $(r_0^s, r_1^s) \in \widetilde{\mathcal{P}_2}$ and we see that $ \dim_{H} \mu_{\varphi} < 1$.

In this case, we see that $a_0 < 1$ and $b_{1} + c_{1} = 0$. 
Hence, this case corresponds to Case 2.2. 
However, we do not need to apply Lemma \ref{lem:classify-multiple} 
since we have the direct upper bound for the relative entropy in \eqref{eq:direct-upper-bound-ent}. 
\end{Exa}
 
\begin{Exa}
Let $N=3$. 
Let 
\[ g_0 (x) \coloneqq \frac{x}{8-x}, \ g_1 (x) \coloneqq \frac{5x+1}{7}, \textup{ and } g_2 (x) \coloneqq \frac{6}{7-x}.  \]
Then, 
\[ \alpha = \min\left\{0,-\frac{1}{7}, 0, -\frac{1}{6}\right\} = -\frac{1}{6} \textup{ and } \beta =  \max\left\{0,-\frac{1}{7}, 0, -\frac{1}{6}\right\} = 0. \]
Hence, $Y = [-1/6, 0]$. 
We see that 
\[ G_0 (y) = \frac{y+1}{y+7}, G_1 (y) = \frac{35(y+1)}{(6y+7)(y+7)}, \textup{ and } G_2 (y) = \frac{1}{6y+7}. \]
Hence, 
\[ \sup_{y \in [-1/6,0]} G_0 (y) = \frac{1}{7},  \textup{ and } \sup_{y \in [-1/6,0]} G_2 (y) = \frac{1}{6}. \]
We also see that 
\[ H_0 (y) = \frac{y-1}{8}, H_1 (y) = \frac{5y}{y+7}, \textup{ and } H_2 (y) = -\frac{1}{6y+7}. \]
Hence, $\textup{Fix}(H_0) = \{-1/7\}$, $\textup{Fix}(H_1) = \{0\}$ and $\textup{Fix}(H_2) = \{-1/6\}$. 
Hence, 
\[ {\bf G} (\textup{Fix}(H_0)) = \left(\frac{1}{8}, \frac{245}{344}, \frac{7}{43}\right),  {\bf G} (\textup{Fix}(H_1)) = \left(\frac{1}{7}, \frac{5}{7}, \frac{1}{7}\right), \textup{ and }  \]
\[ {\bf G} (\textup{Fix}(H_2)) = \left(\frac{5}{41}, \frac{175}{246}, \frac{1}{6}\right). \]
Let 
\[ \widetilde{\mathcal{P}_3} \coloneqq \left\{(p_0,p_1,p_2) \in \overline{\mathcal{P}_3} \ \middle| \ p_0, p_2 \in \left[\frac{7}{16}, \frac{1}{2}\right] \right\}. \]
We now show that   
\begin{equation}\label{eq:ex-3-rel-ent}
s_3 \left({\bf G}(y) | {\bf q}\right) \le -0.74, \ y \in [-1/6,0], \  {\bf q} \in \widetilde{\mathcal{P}_3}. 
\end{equation}

For $q > 0$, let $f_q (x) \coloneqq x \log (q/x)$ for $x \in [0,1]$. 
Then, $f_q$ is increasing on $[0, q/e]$ and decreasing on $[q/e,1]$. 
Furthermore, for each $x \ge 0$, $f_q (x)$ is increasing with respect to $q$. 
Since ${\bf q} = (q_0,q_1,q_2) \in  \widetilde{\mathcal{P}_3}$, it holds that $q_0, q_2 \le 1/2$ and $q_1 \le 1/8$. 
We also see that $\max\{G_0 (y), G_2 (y)\} \le \dfrac{1}{6} < \dfrac{1}{2e}$ and $\dfrac{1}{8e} < \dfrac{2}{3} \le G_1 (y)$.

Therefore, for $y \in [-1/6,0]$, 
\begin{align*} 
s_3 \left({\bf G}(y) | {\bf q}\right) &= f_{q_0}(G_0 (y)) + f_{q_1}(G_1 (y)) + f_{q_2}(G_2 (y)) \\
&\le f_{1/2}(G_0 (y)) + f_{1/8}(G_1 (y)) + f_{1/2}(G_2 (y)) \\
&\le 2f_{1/2}(1/6) + f_{1/8}(2/3) = \log 3 - \frac{8}{3} \log 2 \le -0.74. 
\end{align*} 
Thus we have \eqref{eq:ex-3-rel-ent}. 

By Proposition \ref{prop:reduction-rel-ent}, 

\begin{equation}\label{eq:upper-ex-numerical-2}
\limsup_{n \to \infty} \frac{-\log R_{n}({\bf q};z)}{n}  \le -0.74, \ \textup{  $\nu_0$-a.s. $z$, \ ${\bf q}  \in \widetilde{\mathcal{P}_3}$.} 
\end{equation}

Let $\epsilon \in (0,1/4)$ and let$\{f_0, f_1,f_2\}$ be a contractive system such that $\epsilon \le r_i \le 1-\epsilon$, where $r_i \coloneqq \|f_i\|_{\textup{Lip}}$, $i=0,1,2$.   
Let $s \ge 1$ be the constant such that $r_0^s + r_1^s + r_2^s = 1$. 
Assume that $(r_0^s, r_1^s, r_2^s) \in \widetilde{\mathcal{P}_3}$. 
Then, by \eqref{eq:upper-ex-numerical-2} and Lemma \ref{lem:dimH}, 
\[ \dim_{H} \mu_{\varphi} \le s - \frac{0.74}{\log (1/\epsilon)}. \]

By Lemma \ref{lem:delta-specify}, 
if $\{f_0, f_1, f_2\}$ is a contractive system such that $0.05 \le r_i \le 0.95$, $i=0,1,2$, and $r_0 + r_1 + r_2 \le 1.1$, 
then 
$s \le \dfrac{\log 3}{\log 3 - \log 1.1}$. 

Since we see that 
$\dfrac{0.74}{\log 20} > 0.24$ and $\dfrac{\log 3}{\log 3 - \log 1.1} < 1.1$, 
equation \eqref{eq:final-delta} holds for $N=3$, $\epsilon = 0.05$, $\epsilon_2 = 0.74$ and $\delta = 0.1$. 

Therefore, if $(r_0^s, r_1^s, r_2^s) \in \widetilde{\mathcal{P}_3}$, 
then 
\[ \dim_{H} \mu_{\varphi} \le s - \frac{0.74}{\log (1/\epsilon)} \le \dfrac{\log 3}{\log 3 - \log 1.1} - \frac{0.74}{\log 20}  < 1. \]

Now we can let $\{f_0, f_1,f_2\}$ be a compatible system such that $0.475 \le r_0 \le 0.5$, $0.05 \le r_1 \le 0.1$, and $0.475 \le  r_2 \le 0.5$. 
Then, $1 \le r_0 + r_1 + r_2 \le 1.1$.   
The maps $f_0$, $f_1$ and $f_2$ can be non-affine maps. 
Since $\displaystyle s \le \frac{\log 3}{\log 3 - \log 1.1} \le 1.1$, 
we have $\min\{r_0^s, r_2^s\} \ge (0.475)^{1.1} \ge \frac{7}{16}$. 
Hence, $(r_0^s, r_1^s, r_2^s) \in \widetilde{\mathcal{P}_3}$ and $ \dim_{H} \mu_{\varphi} < 1$.  
\end{Exa}

\begin{Exa}
Let $N = 4$. 
Let 
\[ g_0 (x) = \frac{x}{3}, \ g_1 (x) = \frac{x+2}{6}, \ g_2 (x) = \frac{x+3}{6}, \textup{ and } g_3 (x) = \frac{x+2}{3}.   \]
Then, $\alpha = \beta = 0$, and hence, $Y = \{0\}$. 
We see that 
\[ G_0 (0) = \frac{1}{3}, \ G_1 (0) = \frac{1}{6}, \ G_2 (0) = \frac{1}{6},  \textup{ and }  G_3 (0) = \frac{1}{3}.\]  
Let 
\[ \widetilde{\mathcal{P}_4} \coloneqq \left\{(p_0,p_1,p_2,p_3) \in \overline{\mathcal{P}_4} \ \middle| \ p_0, p_3 \in \left[0, \frac{1}{8}\right] \right\}. \]
Then, for ${\bf q} = (q_0, q_1, q_2, q_3) \in \widetilde{\mathcal{P}_4}$, 
\[ s_4 \left({\bf G}(0) | {\bf q} \right) =  \log 3 + \frac{1}{3} \log 2 + \frac{1}{6} (2\log q_0 + \log q_1 + \log q_2 + 2\log q_3). \]
Since $\displaystyle \sum_{i=0}^{3} q_i = 1$, 
$q_1 + q_2 = 1 - q_0 - q_3$, and hence, 
$\log q_1 + \log q_2 \le 2 \log ((1-q_0-q_3)/2)$. 
Therefore, 
\[ s_4 \left({\bf G}(0) | {\bf q} \right) \le  \log 3 +  \frac{1}{3} \left(\log q_0 + \log q_3 + \log(1-q_0-q_3)\right). \]
Let $f(x,y) \coloneqq \log x + \log y + \log (1-x-y)$ for $x, y \in [0,1/8]$. 
Then, 
$$\max_{x, y \in [0,1/8]} f(x,y) = f\left(\frac{1}{8}, \frac{1}{8}\right) = \log 3 - 8\log 2. $$
Therefore, 
\[ s_4 \left({\bf G}(0) | {\bf q} \right) \le \frac{4}{3} \log \frac{3}{4} \le -0.38, \ {\bf q} \in \widetilde{\mathcal{P}_4}.  \]

By Proposition \ref{prop:reduction-rel-ent}, 

\begin{equation}\label{eq:upper-ex-numerical-3}
\limsup_{n \to \infty} \frac{-\log R_{n}({\bf q};z)}{n}  \le -0.38, \ \textup{  $\nu_0$-a.s.$z$, \ ${\bf q}  \in \widetilde{\mathcal{P}_4}$.} 
\end{equation}

Let $\epsilon \in (0,1/4)$ and let $\{f_0, f_1,f_2,f_3\}$ be a contractive system such that $\epsilon \le r_i \le 1-\epsilon$, where $r_i \coloneqq \|f_i\|_{\textup{Lip}}$, $i=0,1,2,3$.   
Let $s \ge 1$ be the constant such that $r_0^s + r_1^s + r_2^s + r_3^s= 1$. 
Assume that $(r_0^s, r_1^s, r_2^s, r_3^s) \in \widetilde{\mathcal{P}_4}$. 
Then, by \eqref{eq:upper-ex-numerical-3} and Lemma \ref{lem:dimH}, 
\[ \dim_{H} \mu_{\varphi} \le s - \frac{0.38}{\log (1/\epsilon)}. \]

By Lemma \ref{lem:delta-specify}, 
if $\{f_0, f_1, f_2, f_3\}$ is a contractive system such that $0.1 \le r_i \le 0.9$, $i=0,1,2,3$, and $r_0 + r_1 + r_2 + r_3 \le 1.2$, 
then 
$s \le \dfrac{\log 4}{\log 4 - \log 1.2}$. 

Since we see that 
$\dfrac{0.38}{\log 10} > 0.165$ and $1.15 < \dfrac{\log 4}{\log 4 - \log 1.2} < 1.152$, 
equation \eqref{eq:final-delta} holds for $N=4$, $\epsilon = 0.1$, $\epsilon_2 = 0.38$ and $\delta = 0.2$. 

Therefore, if $(r_0^s, r_1^s, r_2^s, r_3^s) \in \widetilde{\mathcal{P}_4}$, 
then 
\[ \dim_{H} \mu_{\varphi} \le s - \frac{0.38}{\log (1/\epsilon)} \le \dfrac{\log 4}{\log 4 - \log 1.2} - \frac{0.38}{\log 10}  < 1. \]

Now we can let $\{f_0, f_1,f_2,f_3\}$ be a compatible system such that $0.1 \le r_0 \le 0.125$, $0.4 \le r_1 \le 0.475$, $0.4 \le r_2 \le 0.475$, $0.1 \le r_3 \le 0.125$. 
Then, $1 \le r_0 + r_1 + r_2 + r_3 \le 1.2$.   
The maps $f_0$, $f_1$, $f_2$ and $f_3$ can be non-affine maps. 
Since $s \ge 1$, we have $\max\{r_0^s, r_3^s\} \le 1/8$. 
Hence, $(r_0^s, r_1^s, r_2^s, r_3^s) \in \widetilde{\mathcal{P}_4}$, and $\dim_H \mu_{\varphi} < 1$. \\
\end{Exa}

{\it Acknowledgments} \  
The author would like to express his gratitude to the reviewer for his or her very careful reading of the manuscript and for providing many valuable comments. 
Example \ref{exa:fixed-point-not-D-system} is suggested by the reviewer.\\ 

\bibliographystyle{plain}
\bibliography{qs-conjugate-dR-revision}

\begin{thebibliography}{10}

\bibitem{Allaart2018}
Pieter~C. Allaart.
\newblock Differentiability and {H}\"older spectra of a class of self-affine
  functions.
\newblock {\em Adv. Math.}, 328:1--39, 2018.

\bibitem{Azuma1967}
Kazuoki Azuma.
\newblock Weighted sums of certain dependent random variables.
\newblock {\em T{\^o}hoku Math. J. (2)}, 19:357--367, 1967.

\bibitem{Baek2011}
In-Soo Baek.
\newblock Derivative of the {R}iesz-{N}\'agy-{T}ak\'acs function.
\newblock {\em Bull. Korean Math. Soc.}, 48(2):261--275, 2011.

\bibitem{Barany2018}
Bal{\'a}zs B{\'a}r{\'a}ny, Gergely Kiss, and Istv{\'a}n Kolossv{\'a}ry.
\newblock Pointwise regularity of parameterized affine zipper fractal curves.
\newblock {\em Nonlinearity}, 31(5):1705--1733, 2018.

\bibitem{BenSlimane2008}
Mourad Ben~Slimane.
\newblock Multifractal formalism for the generalized de {Rham} function.
\newblock {\em Curr. Dev. Theory Appl. Wavelets}, 2(1):45--88, 2008.

\bibitem{Berg2000}
L.~Berg and M.~Kr\"uppel.
\newblock De {R}ham's singular function and related functions.
\newblock {\em Z. Anal. Anwendungen}, 19(1):227--237, 2000.

\bibitem{Buescu2021}
Jorge Buescu and Cristina Serpa.
\newblock Compatibility conditions for systems of iterative functional
  equations with non-trivial contact sets.
\newblock {\em Results Math.}, 76(2):Paper No. 68, 19, 2021.

\bibitem{dACFS2017}
E.~de~Amo, M.~D\'iaz Carrillo, and J.~Fern\'andez-S\'anchez.
\newblock A {S}alem generalised function.
\newblock {\em Acta Math. Hungar.}, 151(2):361--378, 2017.

\bibitem{deRham1956}
Georges de~Rham.
\newblock Sur une courbe plane.
\newblock {\em J. Math. Pures Appl. (9)}, 35:25--42, 1956.

\bibitem{deRham1957}
Georges de~Rham.
\newblock Sur quelques courbes d{\'e}finies par des {\'e}quations
  fonctionnelles.
\newblock Univ. {Politec}. {Torino}, {Rend}. {Sem}. {Mat}. 16, 101--112., 1957.

\bibitem{Fan1999}
Aihua Fan and Ka-Sing Lau.
\newblock Iterated function system and {Ruelle} operator.
\newblock {\em J. Math. Anal. Appl.}, 231(2):319--344, 1999.

\bibitem{Girgensohn1993}
Roland Girgensohn.
\newblock Functional equations and nowhere differentiable functions.
\newblock {\em Aequationes Math.}, 46(3):243--256, 1993.

\bibitem{Girgensohn2006}
Roland Girgensohn, Hans-Heinrich Kairies, and Weinian Zhang.
\newblock Regular and irregular solutions of a system of functional equations.
\newblock {\em Aequationes Math.}, 72(1-2):27--40, 2006.

\bibitem{Hata1985K}
Masayoshi Hata.
\newblock On the functional equation {{\(1/p\cdot \{f(x/p)+\cdots +f((x+p-
  1)/p)\}=\lambda f(\mu x)\)}}.
\newblock {\em J. Math. Kyoto Univ.}, 25:357--364, 1985.

\bibitem{Hata1985}
Masayoshi Hata.
\newblock On the structure of self-similar sets.
\newblock {\em Japan J. Appl. Math.}, 2:381--414, 1985.

\bibitem{Jordan2016}
Thomas Jordan and Tuomas Sahlsten.
\newblock Fourier transforms of {Gibbs} measures for the {Gauss} map.
\newblock {\em Math. Ann.}, 364(3-4):983--1023, 2016.

\bibitem{Kawamura2002}
Kiko Kawamura.
\newblock On the classification of self-similar sets determined by two
  contractions on the plane.
\newblock {\em J. Math. Kyoto Univ.}, 42(2):255--286, 2002.

\bibitem{Kawamura2011}
Kiko Kawamura.
\newblock On the set of points where {Lebesgue}'s singular function has the
  derivative zero.
\newblock {\em Proc. Japan Acad., Ser. A}, 87(9):162--166, 2011.

\bibitem{Kessebohmer2008}
Marc Kesseb{\"o}hmer and Bernd~O. Stratmann.
\newblock Fractal analysis for sets of non-differentiability of {Minkowski}'s
  question mark function.
\newblock {\em J. Number Theory}, 128(9):2663--2686, 2008.

\bibitem{Lesniak2020}
Krzysztof Le\'sniak, Nina Snigireva, and Filip Strobin.
\newblock Weakly contractive iterated function systems and beyond: a manual.
\newblock {\em J. Difference Equ. Appl.}, 26(8):1114--1173, 2020.

\bibitem{Mantica2017}
Giorgio Mantica.
\newblock Minkowski's question mark measure.
\newblock {\em J. Approx. Theory}, 222:74--109, 2017.

\bibitem{Mantica2019}
Giorgio Mantica and Vilmos Totik.
\newblock Regularity of {Minkowski}'s question mark measure, its inverse and a
  class of {IFS} invariant measures.
\newblock {\em J. Lond. Math. Soc., II. Ser.}, 99(3):707--732, 2019.

\bibitem{Minkowski1905}
H.~Minkowski.
\newblock On geometry of numbers.
\newblock In {\em Proceedings of the third international congress of
  mathematicians (ICM 1904), Heidelberg, Germany, August 8--13, 1904}, pages
  164--173. Leipzig: B. G. Teubner, 1905.

\bibitem{Morawiec2018}
Janusz Morawiec and Thomas Z{\"u}rcher.
\newblock On a problem of {Janusz} {Matkowski} and {Jacek} {Weso{{\l}}owski}.
\newblock {\em Aequationes Math.}, 92(4):601--615, 2018.

\bibitem{Morawiec2019}
Janusz Morawiec and Thomas Z{\"u}rcher.
\newblock On a problem of {Janusz} {Matkowski} and {Jacek} {Weso{{\l}}owski}.
  {II}.
\newblock {\em Aequationes Math.}, 93(1):91--108, 2019.

\bibitem{Okamura2019}
Kazuki Okamura.
\newblock Some results for conjugate equations.
\newblock {\em Aequationes Math.}, 93(6):1051--1084, 2019.

\bibitem{Okamura2020}
Kazuki Okamura.
\newblock Hausdorff dimensions for graph-directed measures driven by infinite
  rooted trees.
\newblock {\em Real Anal. Exchange}, 45(1):29--71, 2020.

\bibitem{Rana2002}
Inder~K. Rana.
\newblock {\em An introduction to measure and integration}, volume~45 of {\em
  Graduate Studies in Mathematics}.
\newblock American Mathematical Society, Providence, RI, second edition, 2002.

\bibitem{Serpa2015}
Cristina Serpa and Jorge Buescu.
\newblock Non-uniqueness and exotic solutions of conjugacy equations.
\newblock {\em J. Difference Equ. Appl.}, 21(12):1147--1162, 2015.

\bibitem{Serpa2015N}
Cristina Serpa and Jorge Buescu.
\newblock Piecewise expanding maps and conjugacy equations.
\newblock In {\em Nonlinear maps and their applications. Selected contributions
  from the NOMA 2013 international workshop, Zaragoza, Spain, September 3--4,
  2013}, pages 193--202. Cham: Springer, 2015.

\bibitem{Serpa2017}
Cristina Serpa and Jorge Buescu.
\newblock Constructive solutions for systems of iterative functional equations.
\newblock {\em Constr. Approx.}, 45(2):273--299, 2017.

\bibitem{Tao2011}
Terence Tao.
\newblock {\em An introduction to measure theory}, volume 126 of {\em Graduate
  Studies in Mathematics}.
\newblock American Mathematical Society, Providence, RI, 2011.

\bibitem{Williams1991}
David Williams.
\newblock {\em Probability with martingales}.
\newblock Cambridge etc.: Cambridge University Press, 1991.

\bibitem{YHK1997}
Masaya Yamaguti, Masayoshi Hata, and Jun Kigami.
\newblock {\em Mathematics of fractals}, volume 167 of {\em Translations of
  Mathematical Monographs}.
\newblock American Mathematical Society, Providence, RI, 1997.
\newblock Translated from the 1993 Japanese original by Kiki Hudson.

\bibitem{Zdun2001}
Marek~Cezary Zdun.
\newblock On conjugacy of some systems of functions.
\newblock {\em Aequationes Math.}, 61(3):239--254, 2001.

\end{thebibliography}

\end{document}